\documentclass[psamsfonts]{amsart}

\usepackage{amssymb,amsfonts}
\usepackage[all,arc]{xy}
\usepackage{enumerate}
\usepackage{color}
\usepackage{graphicx}
\usepackage{graphics}
\usepackage{mathrsfs}

\newtheorem{thm}{Theorem}[section]

\newtheorem{prop}[thm]{Proposition}
 \newtheorem{cl}[thm]{Claim}

\theoremstyle{definition}
\newtheorem{defn}[thm]{Definition}

\theoremstyle{remark}
\newtheorem{rem}[thm]{Remark}

\makeatletter
\let\c@equation\c@thm
\makeatother
\numberwithin{equation}{section}

\title{Local coordinate systems on quantum flag manifolds}

\author{Farrokh Razavinia}
\email{f-razavinia@edu.hse.ru \\ farrokh.razavinia@gmail.com}
\address{Department of discrete mathematics; Moscow institute of physics and technology (MIPT); 141702, Russian Federation, Moscow Region, Dolgoprudny,Pervomayskaya street, 32 corp.1.}
\date{June 30, 2015}

\begin{document}

\begin{abstract}

This paper consist of 3 sections. In the first section, we will give a brief introduction to the ''Feigin's homomorphisms'' and will see how they will help us to prove our main and fundamental theorems related to quantum Serre relations and screening operators.\\
In the second section, we will introduce Local integral of motions as the space of invariants of nilpotent
part of quantum affine Lie algebras and will find two and three point invariants  in the case of $U_q(\hat{sl_2}) $ by using Volkov's scheme.\\
In the third section, we will introduce lattice Virasoro algebras as the space of invariants of Borel part $U_q(B_{+})$ of $U_q(g)$ for simple Lie algebra $g$ and will find the set of generators of Lattice Virasoro algebra connected to $sl_2$ and $U_q(sl_2)$\\
And as a new result, we found the set of some generators of lattice Virasoro algebra.

\end{abstract}
\maketitle
 

\section{Introduction}
In this section we will introduce Feigin's homomorphisms and we will see that how they will help us to prove our main and fundamental theorems on screening operators.\\
"Feigin's homomorphisms" was born in his new formulation on quantum Gelfand-Kirillov conjecture, which came on a public view at RIMS in 1992 for the nilpotent part $U_q(\mathfrak{n} )$, that are now known as "Feigin's Conjecture".\\
In that mentioned talk, Feigin proposed the existence of a family of homomorphisms from a quantized enveloping algebra to rings of skew-polynomials. These "homomorphisms" are became very useful tools for to study the fraction field of quantized enveloping algebra. \cite{6}
\subsection*{Feigin's homomorphisms on $U_q(\mathfrak{n}$)}
Here we will briefly try to show that what are Feigin's homomorphisms and how they will guide us to reach and to prove that the screening operators are satisfying in quantum Serre relations.\\

Set $C$ as an arbitrary symmetrizable Cartan matrix of rank $r$, and $\mathfrak{n} = \mathfrak{n}_+$ the
standard maximal nilpotent sub-algebra in the Kac-Moody algebra associated with $C$ (thus, $n$ is generated by the elements $E_1, . . ., E_r$ satisfying in the Serre relations). As always $U_q(n)$
is the quantized enveloping algebra of $n$. And  $A = (A_{ij}) = ( d_i c_{ij} )$ is the symmetric matrix corresponding
to
$C$ for non-zero relatively prime integers
$d_1, . . . , d_n$ such that $d_ia_{ij} = d_ja_{ji}$ for all $i, j$.
And set $g$ as a Kac-Moody Lie algebra attached to $A$, on generators $E_i, F_i, H_i, 1 \leq i \leq n$ .\cite{11} Let us to mention some of the structures related to $g$ that we will use them here:\\
$\hspace*{17pt}$ the triangular decomposition $g = \mathfrak{n}_{-} \oplus \mathfrak{h} \oplus \mathfrak{n}_{+}$;\\
$\hspace*{17pt}$ the dual space $\mathfrak{h}^*$; elements of $\mathfrak{h}^*$
will be referred to as weights;\\
$\hspace*{17pt}$ the root space decomposition $\mathfrak{n}_{\pm} = \oplus_{\alpha \in {\Delta_{\pm}}} g_{\alpha}, g_{\alpha_i}= \mathbb{C}E_i$;\\
$\hspace*{17pt}$ the root lattice $\Lambda \in \mathfrak{h}^* , \{ \alpha_1 , \cdots , \alpha_n  \} \subset \Delta_{+} \subset \mathfrak{h}^*$ being the set of simple
roots;\\
$\hspace*{17pt}$  the invariant bilinear form $\Lambda \times \Lambda \rightarrow \mathbb{Z}$ defined by $<\alpha_i , \alpha_j> = d_i a_{ij}.$ \cite{11}

Set $A_1$ and $A_2$ as a $\Lambda-$ graded associative algebras and define a $q-$ twisted tensor product as the algebra $A_1 \bar{\otimes} A_2$ isomorphic with $A_1 \otimes A_2$ as a linear space with multiplication given by $(a_1 \otimes a_2) \cdot (a'_1 \otimes a'_2) := q^{<{\alpha}'_1 , \alpha_2>} a_1 a'_1 \otimes a_2 a'_2,  $  where ${\alpha}'_1 = deg(a'_1)$ and $\alpha_2 = deg(a_2).$ And by this definition $A_1 \bar{\otimes} A_2$ become a $\Lambda-$ graded algebra.
\begin{prop}
	Set $g$ an arbitrary Kac-Moody algebra, then the map
	\begin{equation}
	\bar{\Delta}:U_{q}^{\pm}(g) \rightarrow U_{q}^{\pm}(g) \bar{\otimes} U_{q}^{\pm}(g)
	\end{equation}
	Such that
	\[  \left\{
	\begin{array}{l l}
	\bar{\Delta}(1) := 1 \otimes 1 \\
	\bar{\Delta} (E_i) := E_i \otimes 1 + 1 \otimes E_i \\
	\bar{\Delta} (F_i) := F_i \otimes 1 + 1 \otimes F_i  \\
	\end{array} \right.\]
	for $1 \leqslant i \leqslant n,$ is a homomorphism of associative algebras. \cite{9}\cite{6}
	
\end{prop}
\begin{rem}
	there is no such map as $U_{q}^{\pm}(g) \rightarrow U_{q}^{\pm}(g) \bar{\otimes} U_{q}^{\pm}(g)$ in the case that $g$ is an associative algebra. \cite{9}
\end{rem}
And as always after defining a co-multiplication, $\bar{\Delta}$, then we can extend it by a iteration as a sequence of maps \cite{1}\\
\begin{equation}
{\bar{\Delta}}^n : U_{q}^{-}(g) \rightarrow U_{q}^{-}(g)^{\otimes n}, ~  ~ n = 2, 3,...
\end{equation}
determined by ${\bar{\Delta}}^2 = \bar{\Delta}, ~  ~  {\bar{\Delta}}^m = (\bar{\Delta} \otimes id) \circ  {\bar{\Delta}}^{n-1}$.\\

Now set $\mathbb{C}[X_i]$ as a ring of polynomials in one variable and by equipping it by grading structure $deg X_i = \alpha_i$ for any simple root $\alpha_i$, we can regard it as a $\Lambda-$ graded.\\
By this grading there will be a morphism of $\Lambda-$ graded associative algebras
\begin{equation}
\phi_i : U_{q}^{-}(g) \rightarrow \mathbb{C}[X_i]: F_j \mapsto \delta_{ij} x_i
\end{equation}
By following this construction for any sequence of simple roots ${\beta}_{i_{1}}, \cdots , {\beta}_{i_{k}},$ there will be a morphism of $\Lambda-$ graded associative algebras
\begin{equation}
({\phi}_{i_{1}} \otimes {\phi}_{i_{k}}) \circ {\bar{\Delta}}^k : U_{q}^{-}(g) \rightarrow \mathbb{C }[X_{1i_{1}}]  \bar{\otimes} \cdots \bar{\otimes} \mathbb{C}[X_{ki_{k}}]
\end{equation}
(the cause of double indexation here is the appearance of $i_j$s more than once in the sequence).\\
And finally, $\mathbb{C}[X_{1i_{1}}]  \bar{\otimes} \cdots \bar{\otimes} \mathbb{C}[X_{ki_{k}}]$ is an algebra of skew polynomials $ \mathbb{C}[X_{1i_{1}}, \cdots , X_{ki_{k}}], $ with $\Lambda-$ grading $X_{s i_{s}} X_{t i_{t}} = q^{<\alpha_{i_s} , \alpha_{i_t} >} X_{t i_{t}} X_{s i_{s}},$ for $s > t.$\\
But let us to simplify it as $X_i X_j = q^{<deg X_i , deg X_j >} X_j X_i$; the one that we will use it always.\\

So very briefly we constructed the already mentioned family of morphisms (Feigin's homomorphisms) from $U_{q}^{-}(g)$ (the maximal nilpotent sub-algebra of a quantum group associated to an arbitrary Kac-Moody algebra) to the algebra of skew polynomials.

\section{ The contribution between Quantum Serre relations and screening operators}

\begin{thm}\label{Thm1}
	Set ${Q}=q^2$ and points $x_1 , \cdots , x_n$ such that $x_i x_j = {Q} x_i x_j$ for $i<j.$ And set $\Sigma^x = x_1 + \cdots +x_n$.
	If ${Q}^N =1$ and $x_{i}^{N} =0$ for some natural number $N$, then we claim that ${(\Sigma^x)}^N = 0$
\end{thm}
\begin{proof}
	It's straightforward, just needs to use q-calculation.
\end{proof}

\subsection{sl(3) case}
As we know,  $M_2= \left[ \begin{array}{cc} 2 & -1 \\ -1 & 2 \end{array} \right] $ is the generalized Cartan matrix for $sl(3)$. Set $M_{q_2}= \left[ \begin{array}{cc} q^2 & q^{-1} \\ q^{-1} & q^2 \end{array} \right] $ and call it Cartan type matrix related to $M_2$.
\begin{thm}\label{Thm2}
	Suppose we have two different types of points $x_i$,  Namely, set
	$(x_{2i-1})_i$, that we will call them of type 1 and $(x_{2i})_i$, that we will call them of type 2 for $i \in I= \{1,2 \}$, and the following $q-$ commutative relations:
	\[  \left\{
	\begin{array}{l l}
	x_{j} x_{j'} = q^2 x_{j'} x_{j}  & \quad \text{if $j < j' ~ and ~ j , j' \in  \{ 1 , 3 \} ~ and ~ j=j'$ }\\
	x_{i} x_{i'} = q^2 x_{i'} x_{i} & \quad \text{if $i < i' ~ and ~ i , i' \in  \{ 2 , 4 \} ~ and ~ i=i'$ }\\
	x_{i} x_{j} = q^{-1} x_{j} x_{i}   & \quad \text{if $ i < j $ }\\
	\end{array} \right.\]
	
	Set $\Sigma_{1}^{x} = \Sigma_{i \in I} x_{2i+1}$ and $\Sigma_{2}^{x} = \Sigma_{i \in I} x_{2i}$. We will call them screening operators.\\
	Then we claim that $\Sigma_{1}^{x} $ and $\Sigma_{2}^{x}$ are satisfying on quantum Serre relations:
	\begin{equation}
	{(\Sigma_{1}^{x})}^{2} \Sigma_{2}^{x} - [2]_q \Sigma_{1}^{x} \Sigma_{2}^{x} \Sigma_{1}^{x} + \Sigma_{2}^{x} {(\Sigma_{1}^{x})}^{2} =0
	\end{equation}
	$$ {(\Sigma_{2}^{x})}^{2} \Sigma_{1}^{x} - [2]_q \Sigma_{2}^{x} \Sigma_{1}^{x} \Sigma_{2}^{x} + \Sigma_{1}^{x} {(\Sigma_{2}^{x})}^{2} =0 $$
	
\end{thm}
\begin{proof}
	It's straightforward, just needs to use q-calculation.
\end{proof}
\begin{thm}\label{Thm3}
	Prove Theorem \ref{Thm2} in a general case, i.e. Set points $X_i \in \{ X_1, \cdots , X_n \}$ and $Y_i \in \{ Y_1, \cdots , Y_n \}$  with the following relations;
	\[  \left\{
	\begin{array}{l l}
	X_{i}  X_{j} = q^2  X_{j}  X_{i} & \quad \text{if $i < j$ }\\
	Y_{i}  Y_{j} = q^2  Y_{j}  Y_{i} & \quad \text{if $i < j$ }\\
	X_{i} Y_{j} = q^{-1} Y_{j} X_{i}  & \quad \text{if $i < j$ }\\
	\end{array} \right.\]
	and the screening operators $\Sigma_{1}^{x} = \Sigma_{i=1}^{k} X_i$ and $\Sigma_{1}^{y}=\Sigma_{j=1}^{k} Y_j$.\\
	We claim that
	$\Sigma_{1}^{x}$ and $\Sigma_{1}^{y}$  are satisfying  in quantum Serre relations.\\
\end{thm}
\begin{proof}
	Proof by induction on $k$.\\
	As we see in theorem \ref{Thm2}, it's true for $k=2$. \\
	Suppose that is true for $k=n$, we will prove that it's true for $k=n+1$.\\
	As we set it out, $\mathfrak{n}$ is a nilpotent Lie algebra, so the Cartan sub-algebra of $\mathfrak{n}$ is equal to $\mathfrak{n}$ with Chevally generators $\Sigma_{1}^{X}$ and $\Sigma_{1}^{y}$ as they are satisfying in quantum Serre relations.\\
	So we can define $U_q(n) := < \Sigma_{1}^{x} , \Sigma_{1}^{y} ~ ~ | ~ {(\Sigma_{1}^{x})}^{2} \Sigma_{1}^{y}  - (q + q^{-1}) \Sigma_{1}^{x} \Sigma_{1}^{y} \Sigma_{1}^{x} + \Sigma_{1}^{y} {(\Sigma_{1}^{x})}^{2}=0 >$ .\\
	Let  $ \mathbb{C}_q[X]$ be the quantum polynomial ring in one variable. We define:\\
	$U_q(n) \bar{\otimes} \mathbb{C}_q[X] := < \Sigma_{1}^{x} , \Sigma_{1}^{y} , X ~ ~ | ~ {(\Sigma_{1}^{x})}^{2} \Sigma_{1}^{y}  - (q + q^{-1}) \Sigma_{1}^{x} \Sigma_{1}^{y} \Sigma_{1}^{x} + \Sigma_{1}^{y} {(\Sigma_{1}^{x})}^{2}=0 , \Sigma_{1}^{x} X = q^2 X \Sigma_{1}^{x} , \Sigma_{1}^{y} X = q^{-1} X \Sigma_{1}^{y} > $.\\
	Here $\bar{\otimes}$ means quantum twisted tensor product.\\
	We define the embedding $U_q(n) \rightarrow U_q(n) \bar{\otimes} \mathbb{C}_q[X] $: $\Sigma_{1}^{X} \mapsto \Sigma_{1}^{x} + X$ ; $\Sigma_{1}^{y} \mapsto \Sigma_{1}^{y}$. \\
	Claim 1:\\
	$(\Sigma_{1}^{x} + X ) ~ and ~ \Sigma_{1}^{y}$ are satisfying on quantum Serre relations.\\
	
	~  ~  ~ \ ~ proof of claim 1:\\

	$(\Sigma_{1}^{x} + X )^2 \Sigma_{1}^{y} - (q + q^{-1})(\Sigma_{1}^{x} + X) \Sigma_{1}^{y} (\Sigma_{1}^{x} + X) + \Sigma_{1}^{y} (\Sigma_{1}^{x} + X)^2 = {(\Sigma_{1}^{x})}^{2} \Sigma_{1}^{y} + \Sigma_{1}^{x} X \Sigma_{1}^{y} + X \Sigma_{1}^{x} \Sigma_{1}^{y} + X^2 \Sigma_{1}^{y} - (q+ q^{-1}) ( \Sigma_{1}^{x} \Sigma_{1}^{y} \Sigma_{1}^{x} + X \Sigma_{1}^{y} \Sigma_{1}^{x} + \Sigma_{1}^{x} \Sigma_{1}^{y} X + X \Sigma_{1}^{y} X ) + \Sigma_{1}^{y} {(\Sigma_{1}^{x})}^{2} + \Sigma_{1}^{y} \Sigma_{1}^{x} X + \Sigma_{1}^{y} X \Sigma_{1}^{x} + \Sigma_{1}^{y} X^2 = {(\Sigma_{1}^{x})}^{2} \Sigma_{1}^{y} - (q+q^{-1}) \Sigma_{1}^{x} \Sigma_{1}^{y} \Sigma_{1}^{x} + \Sigma_{1}^{y} {(\Sigma_{1}^{x})}^{2} + (q^2 X \Sigma_{1}^{x} \Sigma_{1}^{y} + X \Sigma_{1}^{x} \Sigma_{1}^{y} + X^2 \Sigma_{1}^{y} - (q+ q^{-1}) X \Sigma_{1}^{y} \Sigma_{1}^{x} - (q + q^{-1}) q X \Sigma_{1}^{x} \Sigma_{1}^{y} - (q + q^{-1}) q^{-1} X^2 \Sigma_{1}^{y} + q X \Sigma_{1}^{y} \Sigma_{1}^{x} + q^{-1} X \Sigma_{1}^{y} \Sigma_{1}^{x} + q^{-2} X^2 \Sigma_{1}^{y}= 0 + 0 = 0$.\\
	So it's well defined.\\
	Now set $X = X_{n+1}$.\\
	We will have the new operators ${\Sigma_{1}^{x}}' = X_1 + \cdots + X_n + X_{n+1}$ and ${\Sigma_{1}^{y}}' = Y_1 + \cdots + Y_n $.\\
	Now define:\\
	$$U_q(n) \rightarrow U_q(n) \bar{\otimes} \mathbb{C}_q[X] \hookrightarrow  U_q(n) \bar{\otimes} \mathbb{C}_q[X] \bar{\otimes} \mathbb{C}_q[Y]   $$ such that\\
	$$\Sigma_{1}^{x} \longmapsto \Sigma_{1}^{x} + X \longmapsto {\Sigma_{1}^{x}}'$$
	$${\Sigma_{1}^{y}} \longmapsto {\Sigma_{1}^{y}}  \longmapsto {\Sigma_{1}^{y}}' + Y$$
	Notice that $\mathbb{C}_q[X] \bar{\otimes} \mathbb{C}_q[Y] \cong \mathbb{C}<X,Y | XY = q^{-1} YX>$.\\
	And Define:\\
	$U_q(n) \bar{\otimes} \mathbb{C}_q[X,Y] := < \Sigma_{1}^{x} , {\Sigma_{1}^{y}} , X , Y ~~| ~~ \Sigma_{1}^{x} ~ and ~ {\Sigma_{1}^{y}} ~ stisfying ~ q-Serre ~ relations ~ and ~ \Sigma_{1}^{x} X = q^2 X \Sigma_{1}^{x} , \Sigma_{1}^{y} X = q^{-1} X \Sigma_{1}^{y} , \Sigma_{1}^{x} Y = q^{-1} Y \Sigma_{1}^{x}, {\Sigma_{1}^{y}} Y = q^2 Y {\Sigma_{1}^{y}}, XY = q^{-1} Y X >.$\\
	
	Claim 2:\\
	$ {\Sigma_{1}^{x}}' ~ and ~ ({\Sigma_{1}^{y}}' + Y )$ are satisfying on quantum Serre relations.\\

	~  ~  ~ \ ~ proof of claim 2:\\
	${({\Sigma_{1}^{x}}')}^2({\Sigma_{1}^{y}}' +Y) -(q + q^{-1}) {\Sigma_{1}^{x}}' ( {\Sigma_{1}^{y}}' +Y) {\Sigma_{1}^{x}}' + ({\Sigma_{1}^{y}}' + Y ) {({\Sigma_{1}^{x}}')}^2 = {({\Sigma_{1}^{x}}')}^2 {\Sigma_{1}^{y}}' + {({\Sigma_{1}^{x}}')}^2 Y - (q + q^{-1}) ( {\Sigma_{1}^{x}}' {\Sigma_{1}^{y}}' {\Sigma_{1}^{x}}' + {\Sigma_{1}^{x}}' Y {\Sigma_{1}^{x}}' ) + {\Sigma_{1}^{y}}' {({\Sigma_{1}^{x}}')}^2 + Y {({\Sigma_{1}^{x}}')}^2 = {({\Sigma_{1}^{x}}')}^2 {\Sigma_{1}^{y}}' - (q + q^{-1}) {\Sigma_{1}^{x}}' {\Sigma_{1}^{y}}' {\Sigma_{1}^{x}}' + {\Sigma_{1}^{y}}' {({\Sigma_{1}^{x}}')}^2 + {({\Sigma_{1}^{x}}')}^2 Y  + {\Sigma_{1}^{x}}' Y = 0 + q^{-2} Y {({\Sigma_{1}^{x}}')}^2 - (q + q^{-1}) q^{-1} Y {({\Sigma_{1}^{x}}')}^2 + Y {({\Sigma_{1}^{x}}')}^2 = 0 + 0 =0  $.\\
	lets do some part of this computation that maybe make confusion:\\
	$ ({\Sigma_{1}^{x}}')^2 Y = ( {\Sigma_{1}^{x}} + X_{n+1} )^2 Y = (\Sigma_{1}^{x})^2 Y + X_{n+1}^{2} Y+ {\Sigma_{1}^{x}} X_{n+1} Y + X_{n+1} \Sigma_{1}^{x} Y + q^{-2} Y ({\Sigma_{1}^{x}})^{2} + q^{-2} Y X_{n+1}^{2} + q^{-2} Y \Sigma_{1}^{x} X_{n+1} + q^{-2} Y X_{n+1} \Sigma_{1}^{x} = q^{-2} ( Y ( ({\Sigma_{1}^{x}}^{2} +  X_{n+1}^{2} + \Sigma_{1}^{x} X_{n+1} + X_{n+1} \Sigma_{1}^{x} )) = q^{-2} Y {({\Sigma_{1}^{x}}')}^2 . $\\
	And ${\Sigma_{1}^{x}}' Y = ( {\Sigma_{1}^{x}} + X_{n+1} ) Y = {\Sigma_{1}^{x}} Y + X_{n+1} Y = q^{-1} Y {\Sigma_{1}^{x}} + q^{-1} Y X_{n+1} = q^{-1} (Y ( {\Sigma_{1}^{x}} + X_{n+1} )) = q^{-1} Y {\Sigma_{1}^{x}}'$. And by substituting these, we have the result.\\
	So our definition is well defined.\\
	Now set $Y = Y_{n+1}$ and we are done.

\end{proof}
\subsection{ affinized Lie algebra $\hat{sl(2)}$}
As we know,  $\hat{M_2}= \left[ \begin{array}{cc} 2 & -2 \\ -2 & 2 \end{array} \right] $ is the generalized Cartan matrix for $\hat{sl(2)}$. Set $\hat{M_{q_2}}= \left[ \begin{array}{cc} q^2 & q^{-2} \\ q^{-2} & q^2 \end{array} \right] $ and call it Cartan type matrix related to $\hat{M_2}$.\\
$\hat{sl(2)}$ is satisfying in Theorems \ref{Thm2} and \ref{Thm3} as well; but what we need is just  to change the quantum Serre relations in the following case:\\
\begin{equation}\label{Equ1}
{(\Sigma_{1}^{x})}^{3} \Sigma_{1}^{y} - (q^2 + 1 + q^{-2} ) {(\Sigma_{1}^{x})}^{2} \Sigma_{1}^{y} \Sigma_{1}^{x} + ( q^2 + 1 + q^{-2} ) \Sigma_{1}^{x} \Sigma_{1}^{y} {(\Sigma_{1}^{x})}^{2} - \Sigma_{1}^{y} {(\Sigma_{1}^{x})}^{3} = 0
\end{equation}
$${(\Sigma_{1}^{y})}^{3} \Sigma_{1}^{x} - (q^2 + 1 + q^{-2} ) {(\Sigma_{1}^{y})}^{2} \Sigma_{1}^{x} \Sigma_{1}^{y} + ( q^2 + 1 + q^{-2} ) \Sigma_{1}^{y} \Sigma_{1}^{x} {(\Sigma_{1}^{y})}^{2} - \Sigma_{1}^{x} {(\Sigma_{1}^{y})}^{3} = 0$$
And to change the $q-$ commutation relations also; according to our new Cartan type matrix\\
\[  \left\{
\begin{array}{l l}
X_{i}  X_{j} = q^2  X_{j}  X_{i} & \quad \text{if $i < j$ }\\
Y_{i}  Y_{j} = q^2  Y_{j}  Y_{i} & \quad \text{if $i < j$ }\\
X_{i} Y_{j} = q^{-2} Y_{j} X_{i}  & \quad \text{if $i < j$ }
\end{array} \right.\]

But lets try to prove it in the case of Laurent skew $q-$polynomials $\mathbb{C}[X,X^{-1}].$
\begin{thm}
	Set points $X_i \in \{ X_1, \cdots , X_k \}$ and $X_{j}^{-1} \in \{ X_{1}^{-1}, \cdots , X_{k}^{-1} \}$  with the following relations;
	\[  \left\{
	\begin{array}{l l}
	X_{i} X_{j}  = q^{2}  X_{j} X_{i}  & \quad \text{if $i < j$ }\\
	X_{i} X_{j}^{-1} = q^{-2} X_{j}^{-1} X_{i}  & \quad \text{if $i < j$ }
	\end{array} \right.\]
	and the screening operators $\Sigma_{1}^{x} = \Sigma_{i=1}^{k} X_{i}$ and $\Sigma_{1}^{x^{-1}} = \Sigma_{j=1}^{k}X_{j}^{-1}$.\\
	Again we claim that
	$\Sigma_{1}^{x}$ and $\Sigma_{1}^{x^{-1}}$  are satisfying  in quantum Serre relations (\ref{Equ1}).
	
\end{thm}
\begin{proof}
	Proof by induction on $k$.\\
	For $k=2$, Set $\Sigma_{1}^{x} = x_1 + x_2$ and $\Sigma_{1}^{x^{-1}} = x_{1}^{-1} + x_{2}^{-1}$ and as we checked out, it's straightforward to show that they are satisfying in quantum Serre relations (\ref{Equ1}).\\
	Suppose that it's true for $k=n$ components $x_1, \cdots , x_n$. Again as before we define:
	$$U_q(n) := \{ \Sigma_{1}^{x} , \Sigma_{1}^{x^{-1}} | {(\Sigma_{1}^{x})}^{3} \Sigma_{1}^{x^{-1}} - (q^2 + 1 + q^{-2} ) {(\Sigma_{1}^{x})}^{2} \Sigma_{1}^{x^{-1}} \Sigma_{1}^{x} + ( q^2 + 1 + q^{-2} ) \Sigma_{1}^{x^{-1}} \Sigma_{1}^{x} {(\Sigma_{1}^{x^{-1}})}^{2} $$
	$- \Sigma_{1}^{x^{-1}} {(\Sigma_{1}^{x})}^{3} = 0 \} $\\
	Define $U_q(n) \rightarrow \mathbb{C}_q[X,X^{-1}]$; $\Sigma_{1}^{x} \mapsto \lambda X$; $\Sigma_{1}^{x^{-1}} \mapsto X^{-1}$ for $\lambda \in \mathbb{C}^*$\\
	And define $U_q(n) \rightarrow U_q(n) \bar{\otimes} U_q(n) \rightarrow U_q(n) \bar{\otimes} \mathbb{C}[X, X^{-1} ]$; $\Sigma_{1}^{x} \mapsto \Sigma_{1}^{x} \otimes 1 + X \bar{\otimes } \Sigma_{1}^{x}$;$\Sigma_{1}^{x^{-1}} \mapsto \Sigma_{1}^{x^{-1}} \otimes 1 + X^{-1} \bar{\otimes } \Sigma_{1}^{x^{-1}}$.\\
	And $U_q(n) \rightarrow \underbrace{U_q(n) \bar{\otimes} U_q(n) \bar{\otimes} \cdots \bar{\otimes} U_q(n)}_{n ~ terms} \rightarrow \mathbb{C}[X_1, X_{1}^{-1} ] \bar{\otimes} \mathbb{C}[X_2, X_{2}^{-1} ] \bar{\otimes} \cdots \bar{\otimes} \mathbb{C}[X_n, X_{n}^{-1} ] \cong \mathbb{C}[X_ 1, X_{1}^{-1} , \cdots ,X_n, X_{n}^{-1} ]  $
	
\end{proof}

\section{Local integral of motions; Volkov's scheme}

Set two screening operators\\
\begin{equation}
\Sigma_{i}^{x^{-1}} = \Sigma_i x_{i}^{-1},\\
\end{equation}
$$
\Sigma_{j}^{x} = \Sigma_j x_j.
$$
as we already saw, for these operators we have ${(ad_q \Sigma_{i}^{x^{-1}})}^{1-a_{ij}} (\Sigma_{j}^{x})=0$.\\
The project here is to find an analogue of $R-$ matrix $''R''$ such that
\begin{equation}\label{Equ0}
(X_1+ \cdots + X_k) R(X_1,  \cdots  , X_k) = R(X_1,  \cdots  , X_k) (X_1+ \cdots + X_k)
\end{equation}
satisfy.\\

In this section we will try to find a solution for this equation as Volkov planed. We call these kind of solutions  as ''Local integral of motions''.\\
In the sense of Feigin-Pugai \cite{13}, the main idea for to solve such kind of equations is to add ''spectral parameter'' $\beta$ to $k$ points screening operator and to define an analogue of $R-$ matrix:
\begin{equation}\label{Equ2}
(\beta x_1 + x_2 +  \cdots + X_k ) R(x_1,x_2, \cdots  , X_k ) = R(x_1,x_2, \cdots  , X_k )( x_1 + x_2 +  \cdots + \beta X_k ),
\end{equation}
$$
(\beta x_{1}^{-1} + x_{2}^{-1} + \cdots + x_{k}^{-1}) R(x_1,x_2, \cdots  , X_k  ) = R(x_1, x_2, \cdots  , X_k  )( x_{1}^{-1} + x_{2}^{-1} + \cdots + \beta x_{k}^{-1})
$$
\subsection*{ Example $U_q(\hat{sl_2})$; two point invariants}
Let us try to solve this equation for just two points $x_1$ and $x_2$. \\
In this case, our equations (\ref{Equ2}) will reduce to the following ones:
\begin{equation}\label{Equ3}
(\beta x_1 + x_2 ) R(x_1,x_2) = R(x_1,x_2)(x_1 + \beta x_2),
\end{equation}
$$
(\beta x_{1}^{-1} + x_{2}^{-1} ) R(x_1,x_2) = R(x_1,x_2)(x_{1}^{-1} + \beta x_{2}^{-1})
$$
There is a solution for these equation in \cite{2}, but we are interested on re finding them again here. For to do this, let us change the equations (\ref{Equ3}) to the following one, for simplicity. Set $\alpha_1 = x_1 x_{2}^{-1}$ and $R(x_1 , x_2):= R_{1,2}$,
\begin{equation}\label{Equ4}
(\beta x_1 + x_2 ) R(x_1 x_{2}^{-1}) = R(x_1 x_{2}^{-1})(x_1 + \beta x_2),
\end{equation}
$$
(\beta x_{1}^{-1} + x_{2}^{-1} ) R(x_1 x_{2}^{-1}) = R(x_1 x_{2}^{-1})(x_{1}^{-1} + \beta x_{2}^{-1})
$$
The solutions of these equations are identical to the previous one, we did this change just because to find the solutions in these ones are easier than the previous one and less confusing.\\
Then both of (\ref{Equ4}) will reduce to this linear difference equation:
\begin{equation}\label{Equ18}
(\beta \alpha_1 + 1 ) R_{1,2}(q^{-1} u ; \beta) = (u + \beta)R_{1,2}(u ; \beta ),
\end{equation}
Let us do it for one of them (the first one) for to see the procedure (there is an identical approach for the second one):\\
$(\beta x_1 + x_2 ) R(x_1 x_{2}^{-1}) = R(x_1 x_{2}^{-1})(x_1 + \beta x_2)$\\
$(\beta x_1 x_{2}^{-1} + 1 )x_2 R(x_1 x_{2}^{-1}) = R(x_1 x_{2}^{-1})(x_1 x_{2}^{-1} + \beta ) x_2$\\
Set $R(\alpha_1) = \Sigma C_{m0} \alpha_{1}^{m}$ and then distribute $x_2$ in it from the left and then bring it out to the right hand side, by using the $q-$commutation relations. The idea is to disappear $x_2$ from the both sides by multiplying the equation by $x_{2}^{-1}$ from the right side. So we have\\
$(\beta x_1 x_{2}^{-1} + 1 ) R(q^{-1} x_1 x_{2}^{-1}) x_2 = R(x_1,x_2)(x_1 x_{2}^{-1} + \beta ) x_2$ \\
\begin{equation}
(\beta \alpha_1 + 1 ) R_{1,2}(q^{-1} \alpha_1; \beta )= R_{1,2}(\alpha_1 ; \beta)(\alpha_1 + \beta )
\end{equation}

Lets try to find $R_{1,2}(\alpha_1 ; \beta)$:\\

$(\ref{Equ18}) \Rightarrow R_{1,2}(\alpha_1 ; \beta ) = \frac{\beta \alpha_1 + 1 }{\alpha_1 + \beta}  R_{1,2}(q^{-1} \alpha_1 ; \beta) $ \\
$\hspace*{94pt}= \frac{\beta \alpha_1 + 1 }{\alpha_1 + \beta} \cdot \frac{\beta q^{-1} \alpha_1 + 1 }{q^{-1} \alpha_1 + \beta} R_{2,3}(q^{-2} \alpha_1 ; \beta) $\\
$\hspace*{94pt}= \frac{\beta \alpha_1 + 1 }{\alpha_1 + \beta} \cdot \frac{\beta q^{-1} \alpha_1 + 1 }{q^{-1} \alpha_1 + \beta} \cdot \frac{\beta q^{-2} \alpha_1 + 1 }{q^{-2} \alpha_1 + \beta} R_{3,4}(q^{-3} \alpha_1 ; \beta)$ \\
$\hspace*{94pt}= \cdots  = \prod_{i=0}^{\infty} R_{i,i+1} $ \\
$\hspace*{94pt}= \prod_{i=0}^{\infty} \frac{\beta q^{-i} \alpha_1 + 1 }{q^{-i} \alpha_1 + \beta}$\\

But we need some thing more, so lets continue;\\
For to find its recursive sequence we have to pass the following steps:\\

$ \hspace*{80pt}(\beta x_1 + x_2 ) R( \alpha_1 ; \beta) = R( \alpha_1 ; \beta) ( x_1 + \beta x_2 ) $\\
$\hspace*{91pt} ( \beta \alpha_1 + 1 ) R(q^{-1} \alpha_1 ; \beta) = R(\alpha_1 ; \beta) ( \alpha_1 + \beta )$\\
$\Rightarrow_{R(\alpha_1; \beta) = \Sigma_{i=0}^{\infty} C_{i,0} \alpha_{1}^{i}} ~ ~ ~ \Sigma_{i=0}^{\infty} C_{i,0} q^{-i} \beta \alpha_{1}^{i+1}  + \Sigma_{i=0}^{\infty} C_{i,0} q^{-i} \alpha_{1}^i = \Sigma_{i=0}^{\infty} C_{i,0} \alpha_{1}^{i+1} + \Sigma_{i=0}^{\infty} C_{i,0} \beta \alpha_{1}^{i}  $\\

$\hspace*{-5pt}\Sigma_{i=1}^{\infty} C_{i-1,0} q^{-i+1} \beta \alpha_{1}^{i}  + \Sigma_{i=1}^{\infty} C_{i,0} q^{-i} \alpha_{1}^i + C_{0,0} = \Sigma_{i=1}^{\infty} C_{i-1,0} \alpha_{1}^{i} + \Sigma_{i=1}^{\infty} C_{i,0} \beta \alpha_{1}^{i} + \beta C_{0,0}$\\
$ \hspace*{80pt} \Sigma_{i=1}^{\infty} ( (q^{-i + 1} \beta -1 ) C_{i-1 , 0 } + (q^{-i} - \beta ) C_{i, 0 } ) \alpha_{1}^i = 0 $\\
And then by comparing the coefficients in both side of the equation, we reach to the following key rule recursive sequence that we will use it for to find our final solution in the case of two points.
\[  \left\{
\begin{array}{l l}
C_{0,0} = 1 \\
C_{i,0} = \frac{1 -  q^{-i + 1 } \beta  }{q^{-i } - \beta } C_{i - 1,0}  & \quad  ~ for ~ i = 1 , ..., \infty \\
C_{0,0} = \beta C_{0,0} \Rightarrow \beta =1 \\
\end{array} \right.\]
And now let us to set an general agreement for to simplify writing:\\
Set $(\beta)_n := (1- \beta)(1 - q \beta)(1 - q^2 \beta) \cdots ( 1 - q^{n-1} \beta)$ and let our summation be finite, i.e. set $i \in \{ 0 , \cdots , n \}$ and $R(\alpha_1; \beta) = \Sigma_{i=0}^{n} C_{i,0} \alpha_{1}^{i}$ and in the next step we can extend our radius of convergence.\\
Now lets try to find it:\\

$C_{i,0} = \frac{1 -  q^{-i + 1 } \beta  }{q^{-i } - \beta } C_{i - 1,0}$\\
$\hspace*{12pt} C_{i,0} = \frac{1 -  q^{-i + 1 } \beta  }{q^{-i } - \beta } \cdot \frac{1 -  q^{-i + 2 } \beta  }{q^{-i+1 } - \beta } C_{i - 2,0}$\\
$\hspace*{12pt} C_{i,0} = \frac{1 -  q^{-i + 1 } \beta  }{q^{-i } - \beta } \cdot \frac{1 -  q^{-i + 2 } \beta  }{q^{-i+1 } - \beta } \cdot \frac{1 -  q^{-i + 3 } \beta  }{q^{-i+2 } - \beta } C_{i - 3,0}$\\
$\hspace*{12pt} \vdots$\\
$\hspace*{12pt} C_{i,0} = \frac{1 -  q^{-i + 1 } \beta  }{q^{-i } - \beta } \cdot \frac{1 -  q^{-i + 2 } \beta  }{q^{-i+1 } - \beta } \cdot \frac{1 -  q^{-i + 3 } \beta  }{q^{-i+2 } - \beta } \cdots \frac{1 -  q^{-3} \beta  }{q^2 - \beta } \cdot \frac{1 -  q^{-2} \beta  }{q - \beta } \cdot \frac{1 -  q^{-1} \beta  }{1 - \beta } $\\
$\hspace*{12pt} C_{i,0} = \frac{(1 -  q^{-i + 1 } \beta)(1 -  q^{-i + 2 } \beta)(1 -  q^{-i + 3 } \beta ) \cdots (1 -  q^{ -3 } \beta)(1 -  q^{- 2 } \beta)(1 -  q^{-1} \beta)}{(q^{-i } - \beta)(q^{-i+1 } - \beta)(q^{-i+2 } - \beta)(q^{-i+3 } - \beta) \cdots (q^{2 } - \beta)(q - \beta)(1 - \beta)}$\\
$\hspace*{12pt} C_{i,0} = \frac{(1 -  q^{-i + 1 } \beta)(1 -  q^{-i + 2 } \beta)(1 -  q^{-i + 3 } \beta ) \cdots (1 -  q^{ -3 } \beta)(1 -  q^{- 2 } \beta)(1 -  q^{-1} \beta)}{q^{-i }(1 - q^{i }\beta)q^{-i+1 }(1 - q^{i-1 }\beta)q^{-i+2 }(1 - q^{i-2 }\beta)q^{-i+3 }(1 - q^{i-3 } \beta) \cdots q^{-2 }(1 - q^{2 }\beta)q^{-1}(1 - q\beta)(1 - \beta)}$\\
$\hspace*{12pt} C_{i,0} = \frac{(1 -  q^{-i + 1 } \beta)(1 -  q^{-i + 2 } \beta)(1 -  q^{-i + 3 } \beta ) \cdots (1 -  q^{- 3 } \beta)(1 -  q^{ -2 } \beta)(1 -  q^{-1} \beta)}{q^{-i }q^{-i+1 }q^{-i+2 }q^{-i+3 } \cdots q^{-2 } q^{-1}(1 - q^{i }\beta)(1 - q^{i-1 }\beta)(1 - q^{i-2 }\beta)(1 - q^{i-3 } \beta) \cdots (1 - q^{2 }\beta)(1 - q\beta)(1 - \beta)}$\\
$\hspace*{12pt} C_{i,0} = \frac{(1 -  q^{-i + 1 } \beta)(1 -  q^{-i + 2 } \beta)(1 -  q^{-i + 3 } \beta )}{q^{(-i+0)+(-i+1)  +(-i+2) + (-i+3) +  \cdots  + (-i +(i-2)) + (-i +(i-1)) +(-i +(i+0))}}  $\\
$\hspace*{141pt} \cdot \frac{\cdots}{\cdots} \cdot$\\
$\hspace*{42pt} \frac{(1 -  q^{ -3 } \beta)(1 -  q^{ -2 } \beta)(1 -  q^{-1} \beta)}{(1 - q^{i }\beta)(1 - q^{i-1 }\beta)(1 - q^{i-2 }\beta)(1 - q^{i-3 } \beta) \cdots (1 - q^{2 }\beta)(1 - q\beta)(1 - \beta)}$\\

In infinity when $i \rightarrow +\infty$ we have $q^{-i} \rightarrow 1$; So we have:

$\hspace*{12pt} C_{i,0} = \frac{(1 -  q^{-i + 1 } \beta)(1 -  q^{-i + 2 } \beta)(1 -  q^{-i + 3 } \beta )}{q^{(0)+(1)  +(2) + (3) +  \cdots  + ((i-2)) + ((i-1)) +((i+0))}}  $\\
$\hspace*{141pt} \cdot \frac{\cdots}{\cdots} \cdot$\\
$\hspace*{42pt} \frac{(1 -  q^{ -3 } \beta)(1 -  q^{ -2 } \beta)(1 - q^{-1} \beta)}{(1 - q^{i }\beta)(1 - q^{i-1 }\beta)(1 - q^{i-2 }\beta)(1 - q^{i-3 } \beta) \cdots (1 - q^{2 }\beta)(1 - q\beta)(1 - \beta)}$\\
$\hspace*{12pt} C_{i,0} = \frac{(1 -  q^{-i + 1 } \beta)(1 -  q^{-i + 2 } \beta)(1 -  q^{-i + 3 } \beta ) \cdots (1 -  q^{ -3 } \beta)(1 -  q^{ -2 } \beta)(1 -  q^{-1} \beta)}{q^{\frac{n(n-1)}{2}} (q \beta)_n}  $\\

$\hspace*{12pt} C_{i,0} = \frac{(1 -  q^{-i } q\beta)(1 -  q^{-i + 1 } q\beta)(1 -  q^{-i + 2 } q \beta ) \cdots (1 -  q^{ -4 } q\beta)(1 -  q^{ -3 } q\beta)(1 - q^{ -2 } q \beta)}{q^{\frac{n(n-1)}{2}} (q \beta)_n}  $\\

$\hspace*{12pt} C_{i,0} = \frac{(q\beta)(\frac{q^{-1 }}{\beta} -  q^{-i } )(q\beta)(\frac{q^{-1 }}{\beta} -  q^{-i + 1 })(q\beta)(\frac{q^{-1 }}{\beta} -  q^{-i + 2 }  ) \cdots (q\beta)(\frac{q^{-1 }}{\beta} -  q^{ -4 } )(q\beta)(\frac{q^{-1 }}{\beta} -  q^{ -3 } )(q\beta)( \frac{q^{-1}}{\beta}-  q^{ -2 })(1 - \frac{1}{\beta})}{q^{\frac{n(n-1)}{2}} (q \beta)_n (-q \beta)^{-1}}  $\\

$\hspace*{12pt} C_{i,0} = \frac{(q\beta)^{n-1} ( q^{-i } q^{-i + 1 } q^{-i + 2 } \cdots q^{- 2 }q^{ -1 } )( \frac{q^{i - 1 }}{\beta} - 1 )( \frac{q^{i - 2 }}{\beta} -  1 )( \frac{q^{i - 3 }}{\beta}  -  1 ) \cdots ( \frac{q^{3}}{\beta} -  1 )( \frac{q^{2}}{\beta} -  1 )( \frac{q^{1}}{\beta}-  1)(1 - \frac{1}{\beta})}{q^{\frac{n(n-1)}{2}} (q \beta)_n  (-q \beta)^{-1}} $\\

$\hspace*{12pt} C_{i,0} = \frac{(-q\beta)^n ( q^{\frac{n(n-1)}{2}} )(1- \frac{q^{i - 1 }}{\beta}  )(1- \frac{q^{i - 2 }}{\beta} )(1- \frac{q^{i - 3 }}{\beta}  ) \cdots ( 1- \frac{q^{3}}{\beta}  )( 1 -\frac{q^{2}}{\beta}  )(1 - \frac{q}{\beta})(1 - \frac{1}{\beta})}{q^{\frac{n(n-1)}{2}} (q \beta)_n}  $\\
\begin{equation}\label{Equ19}
C_{i,0} = \frac{(-q \beta)^n (\frac{1}{\beta})_n}{(q \beta)_n}
\end{equation}
\subsection*{ Example $U_q(\hat{sl_2})$; three point invariants}

As what we had for previous example in two points; we will proceed the same steps for to find the solution of the equation (\ref{Equ0}) for three points $x_0, x_1 , x_2$.\\
Set $\alpha_i = x_i x_{i+1}^{-1}$; such that  $\alpha_i \alpha_j = q \alpha_j \alpha_i$ for $i , j \in I$ as usual, and $R(\alpha_0, \alpha_1) = \Sigma_{n, m} C_{n, m} \alpha_{0}^{n} \alpha_{1}^{m}$.\\

We are trying to solve the following difference equation subject to $R$;
\begin{equation}
(\beta x_0 +  x_1 + x_2 ) R({\alpha_0}, {\alpha_1} ; \beta) = R({\alpha_0}, {\alpha_1} ; \beta) ( x_0 +  x_1 + \beta x_2 ),
\end{equation}
The process is exactly same as the previous one, so we will skip writing them here.\\
For this equation We got the following recursive sequences, that will guide us to reach to our main solution  for  $n , m = 1 , ..., +\infty$ .


\[  \left\{
\begin{array}{l l}
C_{0,0} = 1 \\
C_{n,m} = \frac{q^{-m + 1 } -  q^{-n - m +2 } \beta }{q^{- m  } - \beta } C_{n - 1,m - 1} +   \frac{1 - q^{-n - m +1 } }{q^{- m  }- \beta } C_{n ,m - 1}   \\
C_{0,m} = \frac{1 - q^{-m + 1 } }{\beta - q^{-m} } C_{0 ,m - 1}  \\

\end{array} \right.\]

And by considering the second sequence as our main key, and following it; we arrived to a nice and important sequence :


\begin{equation}
C_{n,m} = \frac{(- q \beta )^{n} q^{\frac{m(m -1)}{2}} (\frac{1}{\beta})_{n}}{(\beta)_{m} (q \beta)_{{n}- m }}
\end{equation}
That is compatible with the equation (\ref{Equ19}) when $m=0$. And this can show the correctness of our calculation.

\section{Lattice Virasoro algebra}

In this section we are interested on solutions $\Sigma_{1_x}$ of system of difference equation
\[
\begin{cases}
X_i X_j = q X_j X_i       \\
deg(\Sigma_{1_x})=0 \\
[\Sigma_{-\infty}^{+\infty} X_i , \Sigma_{1_x} ]_q =0
\end{cases}
\]
that will be a generator of lattice Virasoro algebra. \\
If we be able to find such kind of solution; then we can extend it to an another generator by a shift operator:
\begin{equation}\label{Equ5}
\Sigma_{2_x} = \Sigma_{1_x} [x_1 \rightarrow x_2 , x_2 \rightarrow x_3, x_3 \rightarrow x_4, \cdots
\end{equation}
$$\Sigma_{3_x} = \Sigma_{2}^{x} [x_2 \rightarrow x_3 , x_3 \rightarrow x_4, x_4 \rightarrow x_5, \cdots$$
$$\vdots$$
where $\Sigma_{1_x} = \Sigma_{1_x} (x_1, x_2, \cdots , x_k)$.
\subsection*{Lattice Virasoro algebra connected to $sl_2$}
Here as always, we have the $q-$commutation relation $X_i X_j = q X_j X_i, ~ i<j$ between the points in $sl_2$.\\
Let us try to find three points invariants; this means to try to solve the following system of difference equation:
\[
\begin{cases}
X_i X_j = q X_j X_i       \\
deg(\Sigma_{1_x})=0 \\
(X_1 + X_2 + X_3) \Sigma_{1_x}(X_1, X_2, X_3) =  \Sigma_{1_x}(X_1, X_2, X_3) (X_1 + X_2 + X_3).
\end{cases}
\]
One can find easily the trivial solutions of the second equation as follows:
\[
\begin{cases}
\Sigma_{11_x}(X_1, X_2, X_3) = X_1 + X_2 + X_3
\Sigma_{12_x}(X_1, X_2, X_3) = X_1  X_{2}^{-1} X_3
\end{cases}
\]
but as we see, non of them have zero grading. So we should find another solution.\\
By just keeping to look at them for a while, we can see that by multiplying these kind of solutions, one can find a zero grading expression, but it's not satisfying for these two ones. Again we note that for a solution; it's inverse is again a solution, so by this remark, we have two options here. We can inverse $\Sigma_{11_x}$ or $\Sigma_{12_x}$ and then multiply it with the other one. In both case we will have same set of generators except that in the first case (inverse of  $\Sigma_{11_x}$), lattice Virasoro algebra is generated by elements of form $\Sigma_{i_x} = X_i X_{i+1}^{-1} X_{i+2} (X_i + X_{i+1} + X_{i+2})^{-1}$ and in the second case (inverse of  $\Sigma_{12_x}$), lattice Virasoro algebra is generated by elements of form $\Sigma_{i_x} = (X_i + X_{i+1} + X_{i+2}) X_{i}^{-1} X_{i+1} X_{i+2}^{-1} $.\\
And by a simple fact that our space of working is closed under multiplication, so these new recently found objects can be a trivial solution for our system of difference equation. And then by shift operators (\ref{Equ5}), we will have the set of generators for our lattice Virasoro algebra connected to $sl_2$.
\subsection*{Lattice Virasoro algebra connected to $U_q(sl_2)$}

Set $A = \frac{\mathbb{C}[q] <x_j , x_i >}{(x_i x_j - q x_j x_i)}$ the algebra of polynomials in variables $q , x_i$  over $\mathbb{C}$ for $i \in I$(our ordered index set), such that
\[
\begin{cases}
q x_i = x_i q & \quad ~ for ~ \text{$i \in I$}\\
x_i x_j = q x_j x_i & \quad  \text{if $i < j$ }
\end{cases}
\]
Our first project is to extend the usual binomial expansion to this algebra, for example we can see the shape of such expansion in a lower exponent 3:\\
$(x_i + x_j)^3 = (x_i + x_j)(x_i + x_j)(x_i + x_j)$\\
$ \hspace*{40pt} = x_i x_i x_i + x_i x_j x_i + x_j x_i x_i + x_j x_j x_i + x_i x_i x_j + x_i x_j x_j + x_j x_i x_j + x_j x_j x_j$\\
$ \hspace*{40pt}= x_i x_i x_i + q x_j x_i x_i +  x_j x_i x_i + x_j x_j x_i + q^2 x_j x_i x_i + q^2 x_j x_j x_i + q x_j x_j x_i + $\\
$ \hspace*{49pt} $ $x_j x_j x_j $\\
$ \hspace*{40pt}= x_{i}^{3} + q x_{j}x_{i}^{2} + x_{j}x_{i}^{2} + x_{j}^{2} x_{i} + q^2 x_{j} x_{i}^{2} + q^2 x_{j}^{2} x_{i} + q  x_{j}^{2} x_{i} + x_{j}^{3} $\\
$ \hspace*{40pt} = x_{i}^{3} + (1 + q + q^2) x_{j}^{2}x_i + (1 + q + q^2) x_{j} x_{i}^{2} + x_{j}^{3}$\\
$ \hspace*{40pt}= \Sigma_{k=0}^{3} {3 \choose k}_q x_{j}^{3-k} x_{i}^{k}$

But for to prove it in a general case, we will use some techniques from combinatorics:\\

$ \hspace*{20pt}$ Suppose $x_j$ and $x_i$ as above and set $\omega$ as a word formed by $x_j$ and $x_i$. Then it's easy to see that any such kind of words can be permuted to $x_{j}^{l} x_{i}^{k}$ along with a factor power of $q$, by using the $q-$commutation rule.  For example, $x_{i} x_j x_j x_i x_i x_j x_j x_i x_i x_j x_i= q^{13} x_{j}^{5} x_{i}^{6} $, as the first $x_i $ should pass 5 $x_j$'s and the second and third $x_i$ should pass 3 $x_j$'s and forth and fifth  $x_i$ should pass 1 $x_j$ and sixth will be stable.\\
Now according to this fact , each word $\omega$ consist of $k$ $x_i$'s and $n-k=l$ $x_j$'s in $(x_i + x_j)^n$. That corresponds to a partition which lies inside an $(n-k) \times k$ rectangle. On the other hand each such partition corresponds to a unique word $\omega$. Lets look at our example again; we have $\omega = x_{i} x_j x_j x_i x_i x_j x_j x_i x_i x_j x_i$, and the partition is $533110$. If $\omega = q^m x_{j}^{n-k} x_{i}^{k},$ then as we see $m$ is the sum of the parts of the partition. And the generating function for all partitions lie inside an $(n-k) \times k$ rectangle is the definition of the $q-$ binomial coefficient ${n \choose n-k}_q = { n \choose k}_q$. So for a positive complex number $n$ we will have
$$ ( x_i + x_j )^n = \Sigma_{k=0}^{n} { n \choose k}_q x_{j}^{n-k} x_{i}^{k}  $$
But what will happen for the negative exponents? \\
It's our next deal for to face. What we need to prove is to see what ${ n \choose k }_q $ will be when we replace $n$ with $-n$?\\
According to the definition $q-$ binomial coefficient, we have  \\ ${ n \choose k }_q = \frac{(1-q^n)(1 - q^{n-1})(1 - q^{n-2}) \cdots (1 - q^{n-k +1})}{(1 - q^{k})(1 - q^{k-1}) \cdots (1 - q^{1})}$.\\
Now by replacing $n$ with $-n$ we will have: \\
${ -n \choose k }_q = \frac{(1-q^{-n})(1 - q^{-n-1})(1 - q^{-n-2}) \cdots (1 - q^{-n-k +1})}{(1 - q^{k})(1 - q^{k-1}) \cdots (1 - q^{1})}$\\
$\hspace*{30pt} = \frac{(1-{q^{-1}}^n)(1 - {q^{-1}}^{n-1})(1 - {q^{-1}}^{n-2}) \cdots (1 - {q^{-1}}^{n+k -1})}{(1 - {q^{-1}}^{-k})(1 - {q^{-1}}^{-k+1}) \cdots (1 - {q^{-1}}^{-1})}$\\
$\hspace*{30pt} = \frac{(1 - {q^{-1}}^{n+k -1}) (1 - {q^{-1}}^{n+k -2})\cdots (1 - {q^{-1}}^{n+k -(k-2)})(1 - {q^{-1}}^{n+k -(k-1)}) (1 - {q^{-1}}^{n+k -(k)})}{(q^{-1})^{-k} ((q^{-1})^{k} -1))(q^{-1})^{-k+1} ((q^{-1})^{k-1} -1)) \cdots (q^{-1})^{-1} ((q^{-1})^{1} -1))}$\\
$\hspace*{30pt} = \frac{(1 - {q^{-1}}^{n+k -1}) (1 - {q^{-1}}^{n+k -2})\cdots (1 - {q^{-1}}^{n+k -(k-2)})(1 - {q^{-1}}^{n+k -(k-1)}) (1 - {q^{-1}}^{n+k -(k)})}{((q^{-1})^{-k} (q^{-1})^{-k+1} \cdots (q^{-1})^{1}) (-1)^k (1 - (q^{-1})^k)(1 - (q^{-1})^{k-1}) \cdots (1 - (q^{-1})^1)} $\\
$\hspace*{30pt} = \frac{(1 - {q^{-1}}^{n+k -1}) (1 - {q^{-1}}^{n+k -2})\cdots (1 - {q^{-1}}^{n+k -(k-2)})(1 - {q^{-1}}^{n+k -(k-1)}) (1 - {q^{-1}}^{n+k -(k)})}{(q^{\frac{k(-k+1)}{2}})(-1)^k ((1 - (q^{-1})^k)(1 - (q^{-1})^{k-1}) \cdots (1 - (q^{-1})^1))}$\\
$\hspace*{30pt} = \frac{(-1)^{-k} q^{\frac{-k(-k+1)}{2}} (1 - {q^{-1}}^{n+k -1}) (1 - {q^{-1}}^{n+k -2})\cdots (1 - {q^{-1}}^{n+k -(k-2)})(1 - {q^{-1}}^{n+k -(k-1)}) (1 - {q^{-1}}^{n+k -(k)}) }{(1 - (q^{-1})^k)(1 - (q^{-1})^{k-1}) \cdots (1 - (q^{-1})^1)}$\\
$\hspace*{30pt} = \frac{(-1)^{k} q^{{k \choose 2}} (1 - {q^{-1}}^{n+k -1}) \cdots (1 - {q^{-1}}^{n}) }{(1 - (q^{-1})^k)(1 - (q^{-1})^{k-1}) \cdots (1 - (q^{-1})^1)}$ \\
$\hspace*{30pt} =  (-1)^{k} q^{{k \choose 2}} {n+k -1 \choose k}_{q^{-1}} $.\\
So as what we had for a positive exponent, we will have the result for negative exponent as follows:
\begin{equation}\label{Equ6}
(x_i + x_j)^{-n} = \Sigma_{k=0}^{\infty} { -n \choose k }_q x_{j}^{-n-k} x_{i}^{k} = \Sigma_{k=0}^{\infty} (-1)^{k} q^{{k \choose 2}} {n+k -1 \choose k}_{q^{-1}}  x_{j}^{-n-k} x_{i}^{k}
\end{equation}

\begin{rem}
	And it's identical to write the summation (\ref{Equ6}) from $-\infty$ to $0$ as follows:
	\begin{equation}
	(x_i + x_j)^{-n} = \Sigma_{k= - \infty}^{0} { -n \choose k }_{q^{-1}} x_{j}^{-n+k} x_{i}^{-k} = \Sigma_{k= - \infty}^{0} (-1)^{k} q^{\frac{k(k+1)}{2}} {n-k -1 \choose -k}_{q}  x_{j}^{-n+k} x_{i}^{-k}
	\end{equation}
\end{rem}

\subsection*{Formulation for to extend to four and more invariant points}

Set ${\Sigma}^x = U_{-} + \Sigma_{i=1}^{k} X_i + U_{+}$, where $U_{+} = \Sigma_{i=k+1}^{+\infty} X_i$ and $U_{-} = \Sigma_{i= -\infty}^{0} X_i$.\\
Set $(F_{1,k}^{x})^{(0)} = f(x_1, \cdots , x_k) = \Sigma C_{\beta} X_{1}^{\beta_1} \cdots  X_{k}^{\beta_k}$ such that\\
$[U_{+} , (F_{1,k}^{x})^{(0)} ] = U_{+} (F_{1,k}^{x})^{(0)} - (F_{1,k}^{x})^{(0)}U_{+}$\\
$\hspace*{64pt} = U_{+} (F_{1,k}^{x})^{(0)} - (F_{1,k}^{x})^{(0)} X_{k+1} -  (F_{1,k}^{x})^{(0)} X_{k+2} - \cdots $\\
$\hspace*{64pt} = U_{+} (F_{1,k}^{x})^{(0)} - q^{- deg (F_{1,k}^{x})^{(0)}} (X_{k+1} + X_{k+2} + \cdots) (F_{1,k}^{x})^{(0)}$\\
$\hspace*{64pt} = U_{+} (F_{1,k}^{x})^{(0)} - q^{- deg (F_{1,k}^{x})^{(0)}} U_{+} (F_{1,k}^{x})^{(0)}  $\\
\begin{equation}
\hspace*{-94pt} = (1 - q^{- deg (F_{1,k}^{x})^{(0)}} )U_{+} (F_{1,k}^{x})^{(0)}
\end{equation}

as well as for $U_{-}$\\
\begin{equation}\label{Equ7}
[U_{-} , (F_{1,k}^{x})^{(0)} ]= (1 - q^{ deg (F_{1,k}^{x})^{(0)}} )U_{-} (F_{1,k}^{x})^{(0)}
\end{equation}

If we suppose $deg (F_{1,k}^{x})^{(0)} = 0$, then both of $[U_{+} , (F_{1,k}^{x})^{(0)} ] $ and $[U_{-} , (F_{1,k}^{x})^{(0)} ]$ will be zero and we will have to check the correctness of $[\Sigma_{i=1}^{k} , (F_{1,k}^{x})^{(0)}  ]\stackrel{?}{=}0$. If it was true? then we will have a generator for lattice Virasoro algebra and by the shift operator, we will have all set of generators for it.\\

Let us define a Poisson bracket as Drinfeld defined and then compute some results by using that:\\
$\{ X_j , X_{i}^{n}  \} := Lim_{q \rightarrow 1} \frac{ad_{X_j} X_{i}^{n}}{1 - q}$ and then we have:\\
$ad_{X_i} X_{i}^{n} =0$, in both  classical and quantum case.\\
$ad_{X_i} X_{j}^{n} = (1 - q^n) X_i X_{j}^{n} $ and $ad_{X_j} X_{i}^{n} = (1 - q^n) X_j X_{i}^{n}$ for $i < j$ in quantum case.\\
$ad_{X_i} X_{j}^{n} = 0 $ and $ad_{X_j} X_{i}^{n} = 0$ for $i < j$ in classical case.\\
Now let us find the Poisson bracket for $X_1$ and $X_{i}^{n}$ in classical case\\
\begin{equation}
\{ X_1 , X_{i}^{n}  \} = Lim_{q \rightarrow 1} \frac{(1 - q^n) X_1 X_{i}^{n}}{1 - q} = - n X_1 X_{i}^{n} \partial_{X_i} ~  ~  ~ for ~ ~ i <1
\end{equation}
\begin{equation}
\{ X_1 , X_{i}^{n}  \} = Lim_{q \rightarrow 1} \frac{(1 - q^{-n}) X_1 X_{i}^{n}}{1 - q} =  n X_1 X_{i}^{n} \partial_{X_i} ~  ~  ~ for ~ ~ i >1
\end{equation}

For example we have $\{ X_1 , X_2 \} = X_1 X_2 \partial_{X_2}$\\
By using this operators in the case of the equation $(\ref{Equ7})$ we see that this part when $q \rightarrow 1$ will be zero and so after this time we just will deal with $\Sigma^{X_i} =  \Sigma_{i=1}^{k} X_i + U_{+} $.\\
And also we can find these relations in a more general case for a one variable function on $X_i$ as follows:\\
And let us consider $E_i , F_i, H_i$ as the generators for $U_q(sl_2)$, then $E_i $ and $H_i$ will produce the Borel part $B_{+}$; One of the ways that we can act $B_{+}$ on the $\mathbb{C} [ X_i , X_{i}^{-1} ]$ is as follows\\
\begin{equation}
\pi : B_{+} \times \mathbb{C} [ X_i , X_{i}^{-1} ] \rightarrow \mathbb{C} [ X_i , X_{i}^{-1} ]: (E_i , P) \mapsto \pi(E_i)P:= ad_{\Sigma^{X}} P = [ \Sigma^{X} , P]_q
\end{equation}
\begin{equation}
\pi : B_{+} \times \mathbb{C} [ X_i , X_{i}^{-1} ] \rightarrow \mathbb{C} [ X_i , X_{i}^{-1} ]: (H_i , P) \mapsto \pi(H_i)P:=< \alpha_i , deg P > P
\end{equation}
where $\alpha_i$ is a simple root related to $H_i$ and $P$ an arbitrary homogeneous element of $\mathbb{C} [ X_i , X_{i}^{-1} ]$
\begin{defn}
	Generators of lattice Virasoro algebra associated to simple Lie algebra $g$ constitute the functional basis of space $\textit{Inv}_{U_q(B_{+}}(\mathbb{C} [ X_i , X_{i}^{-1} ]).$ And for to find these generators we need to solve the following functional equations;[13]\\
	$$[ \Sigma_{i_x} , \Sigma^{X} ] =0 ~ ~ and ~ ~ H_i \Sigma_{i_x} =0 ~  ~  ~  (**)$$
\end{defn}
Now the next question is that "how many variable $X_i$ is enough  for to find a nontrivial solution for equations $(**)?$\\
The answer is if we deal with $q$ in a general position, then one can expect that the dimension is dim$(B_{+}) +1$. [13] So in the case of $sl_2$ it will be $3$, the number of variables.[13]\\

Now let us to go back to our example;\\
$X_1 f(X_0) = f(q^{-1} X_0) X_1$\\
Set $q = exp(H)$; \\
$\hspace*{34pt} = f(e^{-H} X_0) X_1$\\
When $q \rightarrow 1$ then $e^{-H} \rightarrow (1 - H)$ and $e^{H} \rightarrow (1 + H)$;\\
$\hspace*{34pt} = f((1 -H) X_0) X_1 = f(X_0 - H X_0) X_1$\\
$\hspace*{34pt} = (f(X_0) - f(H X_0))X_1 = (f(X_0) - H f(X_0))X_1$\\
$\hspace*{34pt} = (f(X_0) - X_1 H f(X_0)) = (f(X_0) - X_1 X_0 \partial_{X_0} f(X_0))$ \\
$\hspace*{34pt} \Rightarrow \{ X_1 , f(X_0) \} = - X_1 X_0 \partial_{X_0} f(X_0)$
So in general if we repeat the process for any $X_{j < 1}$, we will have\\
\begin{equation}
\{ X_i , f(X_{j < i}) \} = - X_i X_j \partial_{X_j} f(X_{j < i})
\end{equation}
According to the Poisson bracket and our early calculation, equations $(**)$ will have the following form\\
\begin{equation}
E_i \Sigma_{i_x} = (X_1(X_1 + X_2 + X_3+U_{+}) \partial_{X_1} + X_2 (X_2 + X_3 + U_{+}) \partial_{X_2} + X_3 U_{+}  \partial_{X_3} + U_{+}^{2} \partial_{U_{+}}  )\Sigma_{i_x} = 0
\end{equation}
$$H_i \Sigma_{i_x} = (X_1 \partial_{X_1} + X_2 \partial_{X_2} + X_3 \partial_{X_3} + U_{+} \partial_{U_{+}}) \Sigma_{i_x} = 0$$
in three point invariants $X_1, X_2$ and $X_3$.\\

As well as there is a same process with a minor differ for when we have $j >1$ ; \\
\begin{equation}
\{ X_i , f(X_{j > i}) \} =  X_i X_j \partial_{X_j} f(X_{j > i})
\end{equation}
And also $ X_i f(X_{i}) = f(X_{i}) X_i = f(e^H X_{i}) X_i$\\
$\hspace*{60pt} = f((1 + H) X_{i}) X_i = f(X_{i} + H X_{i}) X_i$\\
$\hspace*{60pt} = f(X_{i}) X_i + X_i H f(X_{i}) = f(X_{i}) X_i + X_{i}^{2} \partial_{X_i} f(X_{i}) $\\
\begin{equation}
\hspace*{-60pt} \{ X_{i} , f(X_{i}) \} = X_{i}^{2} \partial_{X_i} f(X_{i})
\end{equation}
Now set $f=f(X_1 , X_2)=(F_{1,2}^{x})^{(-\frac{1}{2})}$, (here $(-\frac{1}{2})$ means that our polynomial is of the degree $-\frac{1}{2}$ ) then by previous definition of $\partial_{U_{+}}$, we have $\partial_{U_{+}} f =0 $.\\

Set $(F_{1,2}^{x})^{(\frac{1}{2})} = [\Sigma^X , (F_{1,2}^{x})^{(-\frac{1}{2})}]_q = ad_{\Sigma^X} (F_{1,2}^{x})^{(-\frac{1}{2})}$, then by using the previous discussion we have $(F_{1,2}^{x})^{(\frac{1}{2})} = (F_{1,2}^{x})^{(\frac{1}{2})} (U_{+} , X_1 , X_2 )$.\\
Now consider the following representation of $(sl_2)_q$;\\
$$F = \partial_{U_{+}}$$
\begin{equation}\label{Equ8}
H = U_{+} \partial_{U_{+}} + X_1 \partial_{X_{1}} + X_2 \partial_{X_{2}}
\end{equation}
$$E = U_{+}^{2} \partial_{U_{+}} + (X_{1}^{2} + X_1 X_2 + X_1 U_{+} ) \partial_{X_{1}} + (X_{2}^{2} + X_2 U_{+})\partial_{X_{2}}$$
We need the highest weight vector of this representation that is the solution of equations (\ref{Equ8}). So we should have $E (F_{1,2}^{x})^{(- \frac{1}{2})} = ad_{\Sigma^X} [\Sigma^X , (F_{1,2}^{x})^{(- \frac{1}{2})} ] =0 $. There is a solution for these equations in \cite{13}. Here we use the same solution and procedure.\\

The Idea is to suppose existing of the local fields [13] $(F_{1,k}^{x})^{(0)}, (F_{2,k+1}^{x})^{(0)}, \cdots$ (here $(0)$ means that the degree is 0) and the modules created from these local fields as follows for degrees $i$ and $j$. According to \cite{13} we try to find the exchange algebra relations
\begin{equation}
(F_{1,k}^{x})^{(i)}= [\Sigma^{X} , [ \Sigma^{X} ,[\cdots , [ \Sigma^{X} , (F_{1,k}^{x})^{(0)}]\cdots]]]  ~  ~  ~  ~ (i ~-times)
\end{equation}
And then again will use  the shift operator and will shift it once for to find another module as follows
$$(F_{2,k+1}^{x})^{(j)}= [\Sigma^{X} , [ \Sigma^{X} ,[\cdots , [ \Sigma^{X} , (F_{2,k+1}^{x})^{(0)}]\cdots]]]  ~  ~  ~  ~ (j ~-times)$$
Let us to proceed as what we planned:\\
Set $k \in \{ 2 \}$ and $i = - \frac{1}{2}$
\begin{equation}
(F_{1,2}^{x})^{(- \frac{1}{2})} = X_{1}^{\frac{1}{2}}X_{2}^{-\frac{1}{2}} (X_{1} + X_{2})^{-\frac{1}{2}}
\end{equation}
as [13]. Then \\
$$[\Sigma^{X} , (F_{1,2}^{x})^{(- \frac{1}{2})}] = [U_{-} + (X_1 + X_2) + U_{+} , X_{1}^{\frac{1}{2}}X_{2}^{-\frac{1}{2}} (X_{1} + X_{2})^{-\frac{1}{2}}] $$
$\hspace*{109pt} =  (1 - q^{-\frac{1}{2}}) U_{+} X_{1}^{\frac{1}{2}}X_{2}^{-\frac{1}{2}} (X_{1} + X_{2})^{-\frac{1}{2}} $ \cite{13}.\\
Where $U_{+}= \Sigma_{3}^{+\infty}$. Let us to call it $(F_{1,2}^{x})^{( \frac{1}{2})}$; because it's degree is $\frac{1}{2}$ and to find another module from it by using shift $X_1 \rightarrow X_3$ and $X_2 \rightarrow X_4$ as follows:\\
\begin{equation}\label{Equ9}
(F_{3,4}^{x})^{(-\frac{1}{2})} = X_{3}^{\frac{1}{2}}X_{4}^{-\frac{1}{2}} (X_{3} + X_{4})^{-\frac{1}{2}}
\end{equation}
And again in a same process we will have\\
\begin{equation}\label{Equ10}
(F_{3,4}^{x})^{(\frac{1}{2})}=[\Sigma^{X} , F_{3,4}^{x})^{(-\frac{1}{2}} ] = (1 - q^{-\frac{1}{2}}) (U_{+}- X_3- X_4) X_{3}^{\frac{1}{2}}X_{4}^{-\frac{1}{2}} (X_{3} + X_{4})^{-\frac{1}{2}} [13]
\end{equation}

Now let us multiply the equations (\ref{Equ9}) and (\ref{Equ10}) (because we need zero degree) with each other for to see what will happen?
$\vspace*{50pt}$
\begin{equation}\label{Equ11}
\hspace*{-140pt} (F_{1,2}^{x})^{(- \frac{1}{2})} (F_{3,4}^{x})^{(\frac{1}{2})}
\end{equation}
$\hspace*{35pt} = X_{1}^{\frac{1}{2}}X_{2}^{-\frac{1}{2}} (X_{1} + X_{2})^{-\frac{1}{2}} ((1 - q^{-\frac{1}{2}}) (U_{+}- X_3- X_4) X_{3}^{\frac{1}{2}}X_{4}^{-\frac{1}{2}} (X_{3} + X_{4})^{-\frac{1}{2}})$\\
$\hspace*{35pt}= X_{1}^{\frac{1}{2}}X_{2}^{-\frac{1}{2}} (X_{1} + X_{2})^{-\frac{1}{2}} U_{+} X_{3}^{\frac{1}{2}}X_{4}^{-\frac{1}{2}} (X_{3} + X_{4})^{-\frac{1}{2}}$ \\
$\hspace*{35pt} - X_{1}^{\frac{1}{2}}X_{2}^{-\frac{1}{2}} (X_{1} + X_{2})^{-\frac{1}{2}} X_{3} X_{3}^{\frac{1}{2}}X_{4}^{-\frac{1}{2}} (X_{3} + X_{4})^{-\frac{1}{2}}$\\
$\hspace*{35pt} - X_{1}^{\frac{1}{2}}X_{2}^{-\frac{1}{2}} (X_{1} + X_{2})^{-\frac{1}{2}} X_{4} X_{3}^{\frac{1}{2}}X_{4}^{-\frac{1}{2}} (X_{3} + X_{4})^{-\frac{1}{2}}$\\
$\hspace*{35pt} -q^{-\frac{1}{2}} X_{1}^{\frac{1}{2}}X_{2}^{-\frac{1}{2}} (X_{1} + X_{2})^{-\frac{1}{2}} U_{+} X_{3}^{\frac{1}{2}}X_{4}^{-\frac{1}{2}} (X_{3} + X_{4})^{-\frac{1}{2}}$ \\
$\hspace*{35pt} -q^{-\frac{1}{2}} X_{1}^{\frac{1}{2}}X_{2}^{-\frac{1}{2}} (X_{1} + X_{2})^{-\frac{1}{2}} X_{3} X_{3}^{\frac{1}{2}}X_{4}^{-\frac{1}{2}} (X_{3} + X_{4})^{-\frac{1}{2}}$\\
$\hspace*{35pt} -q^{-\frac{1}{2}} X_{1}^{\frac{1}{2}}X_{2}^{-\frac{1}{2}} (X_{1} + X_{2})^{-\frac{1}{2}} X_{4} X_{3}^{\frac{1}{2}}X_{4}^{-\frac{1}{2}} (X_{3} + X_{4})^{-\frac{1}{2}}$\\
And again let us proceed as \cite{13} and to find:
\begin{equation}\label{Equ12}
\hspace*{-140pt} -q^{\frac{1}{2}} (F_{1,2}^{x})^{( \frac{1}{2})} (F_{3,4}^{x})^{(- \frac{1}{2})}
\end{equation}
$\hspace*{35pt} = q^{\frac{1}{2}}((1 - q^{-\frac{1}{2}})U_{+}  X_{1}^{\frac{1}{2}}X_{2}^{-\frac{1}{2}} (X_{1} + X_{2})^{-\frac{1}{2}} ) X_{3}^{\frac{1}{2}}X_{4}^{-\frac{1}{2}} (X_{3} + X_{4})^{-\frac{1}{2}} $\\
$\hspace*{35pt} = -q^{\frac{1}{2}} U_{+} X_{1}^{\frac{1}{2}}X_{2}^{-\frac{1}{2}} (X_{1} + X_{2})^{-\frac{1}{2}} X_{3}^{\frac{1}{2}}X_{4}^{-\frac{1}{2}} (X_{3} + X_{4})^{-\frac{1}{2}}  $\\
$\hspace*{35pt} + U_{+} X_{1}^{\frac{1}{2}}X_{2}^{-\frac{1}{2}} (X_{1} + X_{2})^{-\frac{1}{2}} X_{3}^{\frac{1}{2}}X_{4}^{-\frac{1}{2}} (X_{3} + X_{4})^{-\frac{1}{2}}$\\
$\hspace*{35pt} = - X_{1}^{\frac{1}{2}}X_{2}^{-\frac{1}{2}} (X_{1} + X_{2})^{-\frac{1}{2}} U_{+} X_{3}^{\frac{1}{2}}X_{4}^{-\frac{1}{2}} (X_{3} + X_{4})^{-\frac{1}{2}}$\\
$\hspace*{35pt} + q^{-\frac{1}{2}} X_{1}^{\frac{1}{2}}X_{2}^{-\frac{1}{2}} (X_{1} + X_{2})^{-\frac{1}{2}} U_{+} X_{3}^{\frac{1}{2}}X_{4}^{-\frac{1}{2}} (X_{3} + X_{4})^{-\frac{1}{2}}$\\
And again by following \cite{13}, by adding the equations (\ref{Equ11}) and (\ref{Equ12}) for to find an object with zero grading that can be a four point invariant generator of lattice Virasoro algebra\\
\begin{equation}
\rho_{1,4} = (F_{1,2}^{x})^{(- \frac{1}{2})} (F_{3,4}^{x})^{(\frac{1}{2})} -q^{\frac{1}{2}} (F_{1,2}^{x})^{( \frac{1}{2})} (F_{3,4}^{x})^{(- \frac{1}{2})}
\end{equation}
$\hspace*{35pt} = -(1+q^{-\frac{1}{2}})X_{1}^{\frac{1}{2}}X_{2}^{-\frac{1}{2}} (X_{1} + X_{2})^{-\frac{1}{2}} (X_3 + X_4)X_{3}^{\frac{1}{2}}X_{4}^{-\frac{1}{2}} (X_{3} + X_{4})^{-\frac{1}{2}} $\\

But now let us follow precisely the notation from \cite{13} for to not be confused\\
\begin{equation}
\Delta_{1,3} = X_{1}^{\frac{1}{2}}X_{2}^{-\frac{1}{2}} (X_{1} + X_{2})^{-\frac{1}{2}} (X_3 + X_4)X_{3}^{\frac{1}{2}}X_{4}^{-\frac{1}{2}} (X_{3} + X_{4})^{-\frac{1}{2}}
\end{equation}
that is a four point generator of lattice Virasoro algebra, but we are not looking for such kind of solutions, because to extend it to a general form is somehow difficult. So we are looking for a simple solution. \\
Now let us to define another such kind of solutions as we experienced.
\begin{equation}
(F_{2,3}^{x})^{(-\frac{1}{2})} := X_{2}^{\frac{1}{2}}X_{3}^{-\frac{1}{2}} (X_{2} + X_{3})^{-\frac{1}{2}}
\end{equation}
\begin{equation}
(F_{2,3}^{x})^{(\frac{1}{2})} :=(1 - q^{-\frac{1}{2}})(U_{+} - X_3) X_{2}^{\frac{1}{2}}X_{3}^{-\frac{1}{2}} (X_{2} + X_{3})^{-\frac{1}{2}}
\end{equation}
then define
\begin{equation}
\rho_{1,3} = (F_{1,2}^{x})^{(- \frac{1}{2})} (F_{2,3}^{x})^{(\frac{1}{2})} - q^{\frac{1}{2}} (F_{1,2}^{x})^{( \frac{1}{2})} (F_{2,3}^{x})^{(-\frac{1}{2})}
\end{equation}
Let us calculate it;\\
\begin{equation}\label{Equ13}
(F_{1,2}^{x})^{(- \frac{1}{2})} (F_{2,3}^{x})^{(\frac{1}{2})}
\end{equation}
$\hspace*{50pt} = X_{1}^{\frac{1}{2}}X_{2}^{-\frac{1}{2}} (X_{1} + X_{2})^{-\frac{1}{2}} (1 - q^{-\frac{1}{2}})(U_{+} - X_3) X_{2}^{\frac{1}{2}}X_{3}^{-\frac{1}{2}} (X_{2} + X_{3})^{-\frac{1}{2}}$\\
$\hspace*{50pt} = X_{1}^{\frac{1}{2}}X_{2}^{-\frac{1}{2}} (X_{1} + X_{2})^{-\frac{1}{2}} (U_{+} - X_3) X_{2}^{\frac{1}{2}}X_{3}^{-\frac{1}{2}} (X_{2} + X_{3})^{-\frac{1}{2}}   $\\
$\hspace*{50pt} =  q^{-\frac{1}{2}} X_{1}^{\frac{1}{2}}X_{2}^{-\frac{1}{2}} (X_{1} + X_{2})^{-\frac{1}{2}} (U_{+} - X_3) X_{2}^{\frac{1}{2}}X_{3}^{-\frac{1}{2}} (X_{2} + X_{3})^{-\frac{1}{2}}  $\\
$\hspace*{50pt} = X_{1}^{\frac{1}{2}}X_{2}^{-\frac{1}{2}} (X_{1} + X_{2})^{-\frac{1}{2}} U_{+}  X_{2}^{\frac{1}{2}}X_{3}^{-\frac{1}{2}} (X_{2} + X_{3})^{-\frac{1}{2}} $\\
$\hspace*{50pt} - X_{1}^{\frac{1}{2}}X_{2}^{-\frac{1}{2}} (X_{1} + X_{2})^{-\frac{1}{2}} X_{3}  X_{2}^{\frac{1}{2}}X_{3}^{-\frac{1}{2}} (X_{2} + X_{3})^{-\frac{1}{2}} $\\
$\hspace*{50pt} - q^{-\frac{1}{2}} X_{1}^{\frac{1}{2}}X_{2}^{-\frac{1}{2}} (X_{1} + X_{2})^{-\frac{1}{2}} U_{+}  X_{2}^{\frac{1}{2}}X_{3}^{-\frac{1}{2}} (X_{2} + X_{3})^{-\frac{1}{2}} $\\
$\hspace*{50pt} + q^{-\frac{1}{2}} X_{1}^{\frac{1}{2}}X_{2}^{-\frac{1}{2}} (X_{1} + X_{2})^{-\frac{1}{2}} X_{3}  X_{2}^{\frac{1}{2}}X_{3}^{-\frac{1}{2}} (X_{2} + X_{3})^{-\frac{1}{2}} $\\
Now let us calculate the other part
\begin{equation}\label{Equ14}
- q^{\frac{1}{2}} (F_{1,2}^{x})^{( \frac{1}{2})} (F_{2,3}^{x})^{(-\frac{1}{2})}
\end{equation}
$\hspace*{50pt} = - q^{\frac{1}{2}} (1 - q^{-\frac{1}{2}}) U_{+} X_{1}^{\frac{1}{2}}X_{2}^{-\frac{1}{2}} (X_{1} + X_{2})^{-\frac{1}{2}} X_{2}^{\frac{1}{2}}X_{3}^{-\frac{1}{2}} (X_{2} + X_{3})^{-\frac{1}{2}} $\\
$\hspace*{50pt} = - q^{\frac{1}{2}} U_{+} X_{1}^{\frac{1}{2}}X_{2}^{-\frac{1}{2}} (X_{1} + X_{2})^{-\frac{1}{2}} X_{2}^{\frac{1}{2}}X_{3}^{-\frac{1}{2}} (X_{2} + X_{3})^{-\frac{1}{2}} $ \\
$\hspace*{50pt} + U_{+} X_{1}^{\frac{1}{2}}X_{2}^{-\frac{1}{2}} (X_{1} + X_{2})^{-\frac{1}{2}} X_{2}^{\frac{1}{2}}X_{3}^{-\frac{1}{2}} (X_{2} + X_{3})^{-\frac{1}{2}}$\\
$\hspace*{50pt} = - X_{1}^{\frac{1}{2}}X_{2}^{-\frac{1}{2}} (X_{1} + X_{2})^{-\frac{1}{2}} U_{+} X_{2}^{\frac{1}{2}}X_{3}^{-\frac{1}{2}} (X_{2} + X_{3})^{-\frac{1}{2}}$\\
$\hspace*{50pt} + q^{-\frac{1}{2}} X_{1}^{\frac{1}{2}}X_{2}^{-\frac{1}{2}} (X_{1} + X_{2})^{-\frac{1}{2}} U_{+} X_{2}^{\frac{1}{2}}X_{3}^{-\frac{1}{2}} (X_{2} + X_{3})^{-\frac{1}{2}}$\\
And again by following \cite{13}, by adding the equations (\ref{Equ13}) and (\ref{Equ14}), we have another generator of lattice Virasoro algebra, but of three point invariant:
\begin{equation}
\rho_{1,3} = (q^{-\frac{1}{2}} -1)X_{1}^{\frac{1}{2}}X_{2}^{-\frac{1}{2}} (X_{1} + X_{2})^{-\frac{1}{2}} X_{3} X_{2}^{\frac{1}{2}}X_{3}^{-\frac{1}{2}} (X_{2} + X_{3})^{-\frac{1}{2}}
\end{equation}
But now let us follow precisely the notation from \cite{13} for to not be confused\\
\begin{equation}\label{Equ15}
\Delta_{1,2} = X_{1}^{\frac{1}{2}}X_{2}^{-\frac{1}{2}} (X_{1} + X_{2})^{-\frac{1}{2}} X_{3} X_{2}^{\frac{1}{2}}X_{3}^{-\frac{1}{2}} (X_{2} + X_{3})^{-\frac{1}{2}}
\end{equation}
\begin{equation}\label{Equ16}
{(\Delta_{1,3})}^{-1} =  (X_{3} + X_{4})^{\frac{1}{2}} X_{4}^{\frac{1}{2}} X_{3}^{-\frac{1}{2}}  (X_3 + X_4)^{-1}   (X_{1} + X_{2})^{\frac{1}{2}} X_{2}^{\frac{1}{2}} X_{1}^{-\frac{1}{2}}
\end{equation}
is again a solution. \\
Now let us multiply the equations (\ref{Equ15}) and (\ref{Equ16}) for to see what will happen?\\
${(\Delta_{1,3})}^{-1} \Delta_{1,2} =  $\\
$ (X_{3} + X_{4})^{\frac{1}{2}} X_{4}^{\frac{1}{2}} X_{3}^{-\frac{1}{2}}  (X_3 + X_4)^{-1}   (X_{1} + X_{2})^{\frac{1}{2}} X_{2}^{\frac{1}{2}} X_{1}^{-\frac{1}{2}} X_{1}^{\frac{1}{2}}X_{2}^{-\frac{1}{2}} (X_{1} + X_{2})^{-\frac{1}{2}} X_{3} X_{2}^{\frac{1}{2}}X_{3}^{-\frac{1}{2}} (X_{2} + X_{3})^{-\frac{1}{2}} $\\
$= (X_{3} + X_{4})^{\frac{1}{2}} X_{4}^{\frac{1}{2}} X_{3}^{-\frac{1}{2}}  (X_3 + X_4)^{-1} X_3 X_{2}^{\frac{1}{2}}X_{3}^{-\frac{1}{2}} (X_{2} + X_{3})^{-\frac{1}{2}}$\\
by using the equations (\ref{Equ15}) and (\ref{Equ16}) we have the following result which has been mentioned in [13]:
\begin{equation}
\Sigma = (x_3 + x_4)^{-\frac{1}{2}} x_{4}^{\frac{1}{2}} x_{3}^{\frac{1}{2}} (x_2 + x_3)^{-\frac{1}{2}}
\end{equation}

Our next goal is to prove that $\Sigma$ is a generator for lattice Virasoro algebra.\\

Now let us to set some $q-$commutation relations such that any other relation comes from them by using the inverse and multiplication operators.
\[
\begin{cases}
X_i X_j = q X_j X_i       \\
X_{i} X_{j}^{-1} = q^{-1} X_{j}^{-1} X_{i} \\
X_i X_{j}^{-\frac{1}{2}} = q^{-\frac{1}{2}} X_{j}^{-\frac{1}{2}} X_i  \\
X_i X_{j}^{\frac{1}{2}} = q^{\frac{1}{2}} X_{j}^{-\frac{1}{2}} X_i  \\
X_{j}^{\frac{1}{2}} X_{i}^{-1}  = q^{\frac{1}{2}} X_{i}^{-1} X_{j}^{\frac{1}{2}}  \\
X_{j}^{\frac{1}{2}} X_{i}  = q^{-\frac{1}{2}} X_{i} X_{j}^{\frac{1}{2}} \\
X_{j}^{-\frac{1}{2}} X_{i}^{-1}  = q^{-\frac{1}{2}} X_{i}^{-1} X_{j}^{\frac{1}{2}}  \\
X_{j} X_{i}^{-\frac{1}{2}} = q^{\frac{1}{2}} X_{i}^{-\frac{1}{2}} X_{j}\\
X_{i}^{\frac{1}{2}} X_{j}^{\frac{1}{2}}  = q^{\frac{1}{4}} X_{j}^{\frac{1}{2}} X_{i}^{\frac{1}{2}}  \\
X_{i}^{\frac{1}{2}} X_{j}^{-\frac{1}{2}}  = q^{-\frac{1}{4}} X_{j}^{-\frac{1}{2}} X_{i}^{\frac{1}{2}}  \\
X_{i}^{-\frac{1}{2}} X_{j}^{\frac{1}{2}}  = q^{-\frac{1}{4}} X_{j}^{\frac{1}{2}} X_{i}^{-\frac{1}{2}}  \\
X_{j} X_{i}^{\frac{1}{2}}  = q^{-\frac{1}{2}} X_{i}^{\frac{1}{2}} X_{j} \\
X_{i}^{\frac{1}{2}} X_{j}^{-1}  = q^{-\frac{1}{2}} X_{j}^{-1} X_{i}^{\frac{1}{2}} \\
X_{i} X_{j}^{-\frac{1}{2}}  = q^{\frac{1}{2}}  X_{j}^{-\frac{1}{2}} X_{i} \\
\end{cases}
\]

Set $\Sigma^x= \Sigma_{j= -\infty}^{j = +\infty} x_j $. We want to show that $\Sigma^X$ will commute with $\Sigma = (x_3 + x_4)^{-\frac{1}{2}} x_{4}^{\frac{1}{2}} x_{3}^{\frac{1}{2}} (x_2 + x_3)^{-\frac{1}{2}}$ by using the usual commutator$[x , y] = xy - yx$, i.e. to show that the equation $ [\Sigma^x , \Sigma ] \stackrel{?}{=} 0$ is correct.  \\

For to show it, one can easily check
that the contributions of many entries vanishes.
Namely, the elements $X_j$ with indexes $j$  from minus infinity to $1$ and from $5$ to plus infinity
definitely commute due to the rules mentioned above. And after that we can concentrate on the sum $ x_2 + x_3 + x_4 $. \\
So lets do it,\\
$\hspace*{35pt} [x_2 + x_3 + x_4 , (x_3 + x_4)^{-\frac{1}{2}} x_{4}^{\frac{1}{2}} x_{3}^{\frac{1}{2}} (x_2 + x_3)^{-\frac{1}{2}} ] $ \\
$\hspace*{35pt} \stackrel{?}{=} (x_2 + x_3 + x_4)((x_3 + x_4)^{-\frac{1}{2}} x_{4}^{\frac{1}{2}} x_{3}^{\frac{1}{2}} (x_2 + x_3)^{-\frac{1}{2}}) - ((x_3 + x_4)^{-\frac{1}{2}} x_{4}^{\frac{1}{2}} x_{3}^{\frac{1}{2}} (x_2 + $\\
$\hspace*{46pt} x_3)^{-\frac{1}{2}})(x_2 + x_3 + x_4)$\\
for to do this job, let us divide the project to the following small projects. We must demonstrate that the following equations are satisfying.
\begin{equation}\label{Equ17}
(x_2)((x_3 + x_4)^{-\frac{1}{2}} x_{4}^{\frac{1}{2}} x_{3}^{\frac{1}{2}} (x_2 + x_3)^{-\frac{1}{2}}) \stackrel{?}{=} ((x_3 + x_4)^{-\frac{1}{2}} x_{4}^{\frac{1}{2}} x_{3}^{\frac{1}{2}} (x_2 + x_3)^{-\frac{1}{2}})(x_2)
\end{equation}
\begin{equation}
(x_3)((x_3 + x_4)^{-\frac{1}{2}} x_{4}^{\frac{1}{2}} x_{3}^{\frac{1}{2}} (x_2 + x_3)^{-\frac{1}{2}})\stackrel{?}{=}((x_3 + x_4)^{-\frac{1}{2}} x_{4}^{\frac{1}{2}} x_{3}^{\frac{1}{2}} (x_2 + x_3)^{-\frac{1}{2}})(x_3)
\end{equation}
\begin{equation}
(x_4)((x_3 + x_4)^{-\frac{1}{2}} x_{4}^{\frac{1}{2}} x_{3}^{\frac{1}{2}} (x_2 + x_3)^{-\frac{1}{2}})\stackrel{?}{=}((x_3 + x_4)^{-\frac{1}{2}} x_{4}^{\frac{1}{2}} x_{3}^{\frac{1}{2}} (x_2 + x_3)^{-\frac{1}{2}})(x_4)
\end{equation}

Let start with equation $(\ref{Equ17})$:\\
$\hspace*{46pt}  (x_2)((x_3 + x_4)^{-\frac{1}{2}} x_{4}^{\frac{1}{2}} x_{3}^{\frac{1}{2}} (x_2 + x_3)^{-\frac{1}{2}}) \stackrel{?}{=} ((x_3 + x_4)^{-\frac{1}{2}} x_{4}^{\frac{1}{2}} x_{3}^{\frac{1}{2}} (x_2 + x_3)^{-\frac{1}{2}})(x_2)$\\
multiply the equation from the left side with $(x_3 + x_4)^{\frac{1}{2}}$.\\
$\hspace*{46pt} (x_3 + x_4)^{\frac{1}{2}}  x_2 (x_3 + x_4)^{-\frac{1}{2}} x_{4}^{\frac{1}{2}} x_{3}^{\frac{1}{2}} (x_2 + x_3)^{-\frac{1}{2}} \stackrel{?}{=}  x_{4}^{\frac{1}{2}} x_{3}^{\frac{1}{2}} (x_2 + x_3)^{-\frac{1}{2}} x_2 $\\
Now we have two options for $x_2$ to move, it can go to the left direction and act on the $(x_3 + x_4)^{\frac{1}{2}} $ or to the right direction and act on the $(x_3 + x_4)^{-\frac{1}{2}} $. \\
As we see, in $(x_3 + x_4)^{-\frac{1}{2}} $, our summation will take part from $k=0$ to $k= +\infty$. We can use of this fact for to skip some factors in the powers of $q$ that will appear in our calculation when that our partners in action are not two.\\
In the case of $x_2$, in the left hand side we have no problem in our action; because $x_2$ will act on $x_3$ and $x_4$ and we have two different partners in our action. So we can move $x_2$ to the left hand side and act it on $(x_3 + x_4)^{\frac{1}{2}}$:\\
$ (\Sigma_{k=0}^{\frac{1}{2}} {\frac{1}{2} \choose k}_q x_{4}^{\frac{1}{2} - k} x_{3}^{k} )  x_2 (x_3 + x_4)^{-\frac{1}{2}} x_{4}^{\frac{1}{2}} x_{3}^{\frac{1}{2}} (x_2 + x_3)^{-\frac{1}{2}} \stackrel{?}{=}  x_{4}^{\frac{1}{2}} x_{3}^{\frac{1}{2}} (x_2 + x_3)^{-\frac{1}{2}} x_2$\\
$ x_2 (\Sigma_{k=0}^{\frac{1}{2}} {\frac{1}{2} \choose k}_q q^{-\frac{1}{2} - k}x_{4}^{\frac{1}{2} - k} q^{-k} x_{3}^{k} )   (x_3 + x_4)^{-\frac{1}{2}} x_{4}^{\frac{1}{2}} x_{3}^{\frac{1}{2}} (x_2 + x_3)^{-\frac{1}{2}} \stackrel{?}{=}  x_{4}^{\frac{1}{2}} x_{3}^{\frac{1}{2}} (x_2 + x_3)^{-\frac{1}{2}} x_2$\\
$ x_2 q^{-\frac{1}{2}} (\Sigma_{k=0}^{\frac{1}{2}} {\frac{1}{2} \choose k}_q x_{4}^{\frac{1}{2} - k} x_{3}^{k} )   (x_3 + x_4)^{-\frac{1}{2}} x_{4}^{\frac{1}{2}} x_{3}^{\frac{1}{2}} (x_2 + x_3)^{-\frac{1}{2}} \stackrel{?}{=} x_{4}^{\frac{1}{2}} x_{3}^{\frac{1}{2}} (x_2 + x_3)^{-\frac{1}{2}} x_2$\\
$ q^{-\frac{1}{2}} x_2   x_{4}^{\frac{1}{2}} x_{3}^{\frac{1}{2}} (x_2 + x_3)^{-\frac{1}{2}} \stackrel{?}{=}  x_{4}^{\frac{1}{2}} x_{3}^{\frac{1}{2}} (x_2 + x_3)^{-\frac{1}{2}} x_2$\\
Now multiply both side of the equation from the right hand side by $(x_2 + x_3)^{\frac{1}{2}}$; we will have:\\
$ q^{-\frac{1}{2}} x_2   x_{4}^{\frac{1}{2}} x_{3}^{\frac{1}{2}} \stackrel{?}{=}  x_{4}^{\frac{1}{2}} x_{3}^{\frac{1}{2}} (x_2 + x_3)^{-\frac{1}{2}} x_2 (x_2 + x_3)^{\frac{1}{2}}$\\
In the right hand side, we have just one partner in action, i.e. we just have the action of $x_2$ on $x_3$; So there will be a factor of the power of $q$. And as I mentioned it already for to skip this factor we will use of limit in infinity. So $x_2$ should act on $(x_2 + x_3)^{-\frac{1}{2}}$. Lets see what will happen; (here we have to use the equation (\ref{Equ6}) for to expand $(x_2 + x_3)^{-\frac{1}{2}}$):\\
$q^{-\frac{1}{2}} x_2   x_{4}^{\frac{1}{2}} x_{3}^{\frac{1}{2}} \stackrel{?}{=}  x_{4}^{\frac{1}{2}} x_{3}^{\frac{1}{2}} (\Sigma_{k= - \infty}^{0} (-1)^{k} q^{\frac{k(k+1)}{2}} {-\frac{3}{2}-k \choose -k}_{q}  x_{3}^{-\frac{1}{2}+k} x_{2}^{-k} ) x_2 (x_2 + x_3)^{\frac{1}{2}}$\\
$q^{-\frac{1}{2}} x_2   x_{4}^{\frac{1}{2}} x_{3}^{\frac{1}{2}} \stackrel{?}{=}  x_{4}^{\frac{1}{2}} x_{3}^{\frac{1}{2}} (\Sigma_{k= - \infty}^{0} (-1)^{k} q^{\frac{k(k+1)}{2}} {-\frac{3}{2}-k \choose -k}_{q} q^{\frac{1}{2} - k}x_{3}^{-\frac{1}{2}+k} x_{2}^{-k} ) x_2 (x_2 + x_3)^{\frac{1}{2}}$\\
By using the fact that when $k \rightarrow +\infty$ then $q^{\frac{1}{2} - k} \rightarrow q^{\frac{1}{2}}$ we have:\\
$q^{-\frac{1}{2}} x_2   x_{4}^{\frac{1}{2}} x_{3}^{\frac{1}{2}} \stackrel{?}{=}  x_{4}^{\frac{1}{2}} x_{3}^{\frac{1}{2}} q^{\frac{1}{2}}  x_2 (x_2 + x_3)^{\frac{1}{2}}$\\
$q^{-\frac{1}{2}} x_2   x_{4}^{\frac{1}{2}} x_{3}^{\frac{1}{2}} \stackrel{?}{=} q^{\frac{1}{2}} x_{4}^{\frac{1}{2}} x_{3}^{\frac{1}{2}}   x_2 (x_2 + x_3)^{\frac{1}{2}}$\\
$q^{-\frac{1}{2}} x_2   x_{4}^{\frac{1}{2}} x_{3}^{\frac{1}{2}} \stackrel{?}{=} q^{\frac{1}{2}} x_2 q^{-\frac{1}{2}} x_{4}^{\frac{1}{2}} q^{-\frac{1}{2}} x_{3}^{\frac{1}{2}}    (x_2 + x_3)^{\frac{1}{2}}$\\
$q^{-\frac{1}{2}} x_2   x_{4}^{\frac{1}{2}} x_{3}^{\frac{1}{2}} = q^{-\frac{1}{2}} x_2  x_{4}^{\frac{1}{2}}  x_{3}^{\frac{1}{2}}   $\\
So it's correct in the case of $x_2$.\\
The case of $x_3$ and $x_4$ are almost identical to the case of $x_2$ with just some minor differs.\\

Lets do it for $x_3$;\\
$(x_3)((x_3 + x_4)^{-\frac{1}{2}} x_{4}^{\frac{1}{2}} x_{3}^{\frac{1}{2}} (x_2 + x_3)^{-\frac{1}{2}})\stackrel{?}{=}((x_3 + x_4)^{-\frac{1}{2}} x_{4}^{\frac{1}{2}} x_{3}^{\frac{1}{2}} (x_2 + x_3)^{-\frac{1}{2}})(x_3)$\\
multiply the equation with $(x_3 + x_4)^{\frac{1}{2}}$ from the left hand side;\\
$(x_3 + x_4)^{\frac{1}{2}} x_3(x_3 + x_4)^{-\frac{1}{2}} x_{4}^{\frac{1}{2}} x_{3}^{\frac{1}{2}} (x_2 + x_3)^{-\frac{1}{2}} \stackrel{?}{=} x_{4}^{\frac{1}{2}} x_{3}^{\frac{1}{2}} (x_2 + x_3)^{-\frac{1}{2}}x_3$\\
$(x_3 + x_4)^{\frac{1}{2}} x_3(\Sigma_{k=0}^{+ \infty} (-1)^{k} q^{{k \choose 2}} {-\frac{3}{2} \choose k}_{q^{-1}} x_{4}^{-\frac{1}{2} - k} x_{3}^{k} ) x_{4}^{\frac{1}{2}} x_{3}^{\frac{1}{2}} (x_2 + x_3)^{-\frac{1}{2}} \stackrel{?}{=} x_{4}^{\frac{1}{2}} x_{3}^{\frac{1}{2}} (x_2 + x_3)^{-\frac{1}{2}}x_3$\\
$(x_3 + x_4)^{\frac{1}{2}} (\Sigma_{k=0}^{+ \infty} (-1)^{k} q^{{k \choose 2}} {-\frac{3}{2} \choose k}_{q^{-1}} q^{-\frac{1}{2} - k}x_{4}^{-\frac{1}{2} - k} q^k x_{3}^{k} ) x_3 x_{4}^{\frac{1}{2}} x_{3}^{\frac{1}{2}} (x_2 + x_3)^{-\frac{1}{2}} \stackrel{?}{=} x_{4}^{\frac{1}{2}} x_{3}^{\frac{1}{2}} (x_2 + x_3)^{-\frac{1}{2}}x_3$\\
By using the fact that when $k \rightarrow +\infty$ then $q^{-\frac{1}{2} - k} \rightarrow q^{-\frac{1}{2}}$ we have:\\
$q^{-\frac{1}{2}} x_3 x_{4}^{\frac{1}{2}} x_{3}^{\frac{1}{2}} (x_2 + x_3)^{-\frac{1}{2}} \stackrel{?}{=} x_{4}^{\frac{1}{2}} x_{3}^{\frac{1}{2}} (x_2 + x_3)^{-\frac{1}{2}}x_3$\\
multiply the equation with $(x_2 + x_3)^{\frac{1}{2}}$ from the right hand side;\\
$q^{-\frac{1}{2}} x_3 x_{4}^{\frac{1}{2}} x_{3}^{\frac{1}{2}}  \stackrel{?}{=} x_{4}^{\frac{1}{2}} x_{3}^{\frac{1}{2}} (x_2 + x_3)^{-\frac{1}{2}}x_3 (x_2 + x_3)^{\frac{1}{2}}$\\
$q^{-\frac{1}{2}} x_3 x_{4}^{\frac{1}{2}} x_{3}^{\frac{1}{2}}  \stackrel{?}{=} x_{4}^{\frac{1}{2}} x_{3}^{\frac{1}{2}} (\Sigma_{k=0}^{+ \infty} (-1)^{k} q^{{k \choose 2}} {-\frac{3}{2} \choose k}_{q^{-1}} x_{3}^{-\frac{1}{2} - k} x_{2}^{k} )x_3 (x_2 + x_3)^{\frac{1}{2}}$\\
here again we have to change our equation to  (\ref{Equ6});\\
$q^{-\frac{1}{2}} x_3 x_{4}^{\frac{1}{2}} x_{3}^{\frac{1}{2}}  \stackrel{?}{=} x_{4}^{\frac{1}{2}} x_{3}^{\frac{1}{2}}  (\Sigma_{k= - \infty}^{0} (-1)^{k} q^{\frac{k(k+1)}{2}} {-\frac{3}{2}-k \choose -k}_{q}  x_{3}^{-\frac{1}{2}+k} x_{2}^{-k} ) x_3 (x_2 + x_3)^{\frac{1}{2}}$\\
$q^{-\frac{1}{2}} x_3 x_{4}^{\frac{1}{2}} x_{3}^{\frac{1}{2}}  \stackrel{?}{=} x_{4}^{\frac{1}{2}} x_{3}^{\frac{1}{2}} x_3 (\Sigma_{k= - \infty}^{0} (-1)^{k} q^{\frac{k(k+1)}{2}} {-\frac{3}{2}-k \choose -k}_{q}  x_{3}^{-\frac{1}{2}+k} q^{-k} x_{2}^{-k} )  (x_2 + x_3)^{\frac{1}{2}}$\\
By using the fact that when $k \rightarrow +\infty$ then $q^{ - k} \rightarrow 1$ we have:\\
$q^{-\frac{1}{2}} x_3 x_{4}^{\frac{1}{2}} x_{3}^{\frac{1}{2}}  \stackrel{?}{=} x_{4}^{\frac{1}{2}} x_{3}^{\frac{1}{2}} x_3 $\\
$q^{-\frac{1}{2}} x_3 x_{4}^{\frac{1}{2}} x_{3}^{\frac{1}{2}}  = q^{-\frac{1}{2}} x_3 x_{4}^{\frac{1}{2}} x_{3}^{\frac{1}{2}}  $\\

Lets do it for $x_4$;\\
$(x_4)((x_3 + x_4)^{-\frac{1}{2}} x_{4}^{\frac{1}{2}} x_{3}^{\frac{1}{2}} (x_2 + x_3)^{-\frac{1}{2}})\stackrel{?}{=}((x_3 + x_4)^{-\frac{1}{2}} x_{4}^{\frac{1}{2}} x_{3}^{\frac{1}{2}} (x_2 + x_3)^{-\frac{1}{2}})(x_4)$\\
multiply the equation with $(x_3 + x_4)^{\frac{1}{2}}$ from the left hand side;\\
$(x_3 + x_4)^{\frac{1}{2}} x_4(x_3 + x_4)^{-\frac{1}{2}} x_{4}^{\frac{1}{2}} x_{3}^{\frac{1}{2}} (x_2 + x_3)^{-\frac{1}{2}} \stackrel{?}{=}  x_{4}^{\frac{1}{2}} x_{3}^{\frac{1}{2}} (x_2 + x_3)^{-\frac{1}{2}} x_4$\\
$(x_3 + x_4)^{\frac{1}{2}} x_4 (\Sigma_{k=0}^{+ \infty} (-1)^{k} q^{{k \choose 2}} {-\frac{3}{2} \choose k}_{q^{-1}} x_{4}^{-\frac{1}{2} - k} x_{3}^{k} ) x_{4}^{\frac{1}{2}} x_{3}^{\frac{1}{2}} (x_2 + x_3)^{-\frac{1}{2}} \stackrel{?}{=}  x_{4}^{\frac{1}{2}} x_{3}^{\frac{1}{2}} (x_2 + x_3)^{-\frac{1}{2}} x_4$\\
$(x_3 + x_4)^{\frac{1}{2}}  (\Sigma_{k=0}^{+ \infty} (-1)^{k} q^{{k \choose 2}} {-\frac{3}{2} \choose k}_{q^{-1}} x_{4}^{-\frac{1}{2} - k} q^{-k} x_{3}^{k} ) x_4 x_{4}^{\frac{1}{2}} x_{3}^{\frac{1}{2}} (x_2 + x_3)^{-\frac{1}{2}} \stackrel{?}{=}  x_{4}^{\frac{1}{2}} x_{3}^{\frac{1}{2}} (x_2 + x_3)^{-\frac{1}{2}} x_4$\\
By using the fact that when  $k \rightarrow +\infty$ then $q^{ - k} \rightarrow 1$ we have:\\
$x_4 x_{4}^{\frac{1}{2}} x_{3}^{\frac{1}{2}} (x_2 + x_3)^{-\frac{1}{2}} \stackrel{?}{=}  x_{4}^{\frac{1}{2}} x_{3}^{\frac{1}{2}} (x_2 + x_3)^{-\frac{1}{2}} x_4$\\
multiply the equation with $(x_2 + x_3)^{\frac{1}{2}}$ from the left hand side;\\
$x_4 x_{4}^{\frac{1}{2}} x_{3}^{\frac{1}{2}}  \stackrel{?}{=}  x_{4}^{\frac{1}{2}} x_{3}^{\frac{1}{2}} (x_2 + x_3)^{-\frac{1}{2}} x_4 (x_2 + x_3)^{\frac{1}{2}}$\\
$x_4 x_{4}^{\frac{1}{2}} x_{3}^{\frac{1}{2}}  \stackrel{?}{=}  x_{4}^{\frac{1}{2}} x_{3}^{\frac{1}{2}} (x_2 + x_3)^{-\frac{1}{2}} x_4 (\Sigma_{k=0}^{\frac{1}{2}} {\frac{1}{2} \choose k}_q x_{3}^{\frac{1}{2} - k} x_{2}^{k} )$\\
$x_4 x_{4}^{\frac{1}{2}} x_{3}^{\frac{1}{2}}  \stackrel{?}{=}  x_{4}^{\frac{1}{2}} x_{3}^{\frac{1}{2}} (x_2 + x_3)^{-\frac{1}{2}}  (\Sigma_{k=0}^{\frac{1}{2}} {\frac{1}{2} \choose k}_q q^{-\frac{1}{2} + k} x_{3}^{\frac{1}{2} - k} q^{-k} x_{2}^{k} ) x_4$\\
$x_4 x_{4}^{\frac{1}{2}} x_{3}^{\frac{1}{2}}  \stackrel{?}{=}  x_{4}^{\frac{1}{2}} x_{3}^{\frac{1}{2}} (x_2 + x_3)^{-\frac{1}{2}} q^{-\frac{1}{2}} (\Sigma_{k=0}^{\frac{1}{2}} {\frac{1}{2} \choose k}_q  x_{3}^{\frac{1}{2} - k}  x_{2}^{k} ) x_4$\\
$x_4 x_{4}^{\frac{1}{2}} x_{3}^{\frac{1}{2}}  \stackrel{?}{=}  x_{4}^{\frac{1}{2}} x_{3}^{\frac{1}{2}}q^{-\frac{1}{2}}  x_4$\\
$x_4 x_{4}^{\frac{1}{2}} x_{3}^{\frac{1}{2}}  \stackrel{?}{=} q^{-\frac{1}{2}} x_4 x_{4}^{\frac{1}{2}} q^{\frac{1}{2}} x_{3}^{\frac{1}{2}}  $\\
$x_4 x_{4}^{\frac{1}{2}} x_{3}^{\frac{1}{2}}  =  x_4 x_{4}^{\frac{1}{2}} x_{3}^{\frac{1}{2}}  $  \\
And then again by using the shift operators (\ref{Equ5}), we will have the set of generators $\Sigma_{i_x}$ for our lattice Virasoro algebra connected to $U_q(sl_2)$.
\subsection*{Generators of lattice Virasoro algebra coming from 2-dimensional representation of $sl_2$}
\begin{cl}
	$[\Sigma_{j= -\infty}^{j = +\infty} x_j , (x_3 + x_4)^{-1} x_{4} x_{3} (x_2 + x_3)^{-1}]=0$
\end{cl}
\begin{proof}
	The proof is identical to the proof for fractional exponent $\frac{1}{2}$,above.
\end{proof}
\begin{cl}
	$[\Sigma_{j= -\infty}^{j = +\infty} x_j , (x_2 + x_3 + x_4)^{-1} (x_3 + x_4)  x_{2} (x_1 + x_2)^{-1}]=0$
\end{cl}
\begin{proof}
	The proof is identical to the proof for fractional exponent $\frac{1}{2}$. Just what you need is to set $x_3 + x_4 = x'_{3}$. The rest is exactly identical.
\end{proof}
\begin{cl}
	$[\Sigma_{j= -\infty}^{j = +\infty} x_j , (x_2 + \cdots + x_k)^{-1} (x_3 + \cdots + x_k)  x_{2} (x_1 + x_2)^{-1}]=0$
\end{cl}
\begin{proof}
	Proof by induction on $k$.\\
	It is true for $k=3$.\\
	Suppose that it is true for $k-1$ component. Then for $k$ component we have:
	$[\Sigma_{j= -\infty}^{j = +\infty} x_j , ((x_2 + \cdots + x_{k-1}) + x_k)^{-1} ((x_3 + \cdots + x_{k-1}) + x_k)  x_{2} (x_1 + x_2)^{-1}]$\\
	Set $(x_2 + \cdots + x_{k-1}) = x'_{k-1}$;\\
	$\hspace*{35pt} = [\Sigma_{j= -\infty}^{j = +\infty} x_j , ( x'_{k-1} + x_k)^{-1} ( x'_{k-1} + x_k)  x_{2} (x_1 + x_2)^{-1}]$;\\
	So the rest will come from $k=3$ and we are done.
\end{proof}

And then by using the shift operators (\ref{Equ5}), we will have the set of all generators.
\subsection*{Results; Generators of lattice Virasoro algebra coming from 3 and 4-dimensional representation of $sl_2$}
$\vspace*{5pt}$\\
Let us suppose the following $3-$dimensional representation of $sl_2$. The process of defining this representation is the same as $(\ref{Equ8})$.\\
Define:\\
$$F = \partial_{(U_{+} - X_3)}$$
\begin{equation}
H = U_{+} \partial_{U_{+}} + X_1 \partial_{X_{1}} + X_2 \partial_{X_{2}} + X_3 \partial_{X_{3}}
\end{equation}
$$E = (U_{+} - X_3)^{2} \partial_{(U_{+} - X_3)} + (X_{1}^{2} + X_1 X_2 + X_1 X_3 + X_1 (U_{+} - X_3) )\partial_{X_{1}} $$
$$\hspace*{-30pt}+ (X_{2}^{2} + X_2 X_3 + X_2 (U_{+} - X_3))\partial_{X_{2}} + (X_{3}^{2} + X_3 (U_{+} - X_3) )\partial_{X_{3}}$$
As before set\\
$$(F_{k,k+1,k+2}^{x})^{(i)}=[\Sigma^X , [\Sigma^X , [ \cdots , [ \Sigma^X , (F_{1,k,k+1}^{x})^{(i)}]]\cdots ]]$$
$$(F_{k+1,k+2,k+3}^{x})^{(j)}=[\Sigma^X , [\Sigma^X , [ \cdots , [ \Sigma^X , (F_{2,k,k+1}^{x})^{(j)}]]\cdots ]$$
Set $k=1$ and $i=-\frac{1}{2}$ and set $(F_{1,2,3}^{x})^{(-\frac{1}{2})}= X_{1}^{\frac{1}{2}} X_{2}^{-\frac{1}{2}} (X_2 + X_3)^{-\frac{1}{2}}$; because it's satisfying in the relation $[\Sigma^X,[\Sigma^X,(F_{1,2,3}^{x})^{(-\frac{1}{2})}]]=0$ and is our highest vector in this representation.\\
Then \\
$[\Sigma^X,(F_{1,2,3}^{x})^{(-\frac{1}{2})}] = [ U_{-} + (X_1 + X_2 + X_3) + (U_{+} - X_3), X_{1}^{\frac{1}{2}} X_{2}^{-\frac{1}{2}} (X_2 + X_3)^{-\frac{1}{2}}]$\\
$\hspace*{35pt} = (1 - q^{-\frac{1}{2}})(U_{+} - X_3)X_{1}^{\frac{1}{2}} X_{2}^{-\frac{1}{2}} (X_2 + X_3)^{-\frac{1}{2}} $\\
where as usual $U_{+}=\Sigma_{i=3}^{+\infty} X_i$. We call it as usual $(F_{1,2,3}^{x})^{(\frac{1}{2})},$ because it has degree $\frac{1}{2}$.\\
Set $X_1 \rightarrow X_3$ and $X_2 \rightarrow X_4$ and $X_3 \rightarrow X_5$ in $(F_{1,2,3}^{x})^{(-\frac{1}{2})}$, then we will have,\\
\begin{equation}
(F_{3,4,5}^{x})^{(-\frac{1}{2})} = X_{3}^{\frac{1}{2}} X_{4}^{-\frac{1}{2}} (X_4 + X_5)^{-\frac{1}{2}}
\end{equation}
and again in a same process we have,\\
$(F_{3,4,5}^{x})^{(-\frac{1}{2})}= [\Sigma^X , (F_{3,4,5}^{x})^{(-\frac{1}{2})}] $\\
$\hspace*{35pt} = (1 - q^{-\frac{1}{2}})(U_{+} - X_3 - X_4 - X_5)X_{3}^{\frac{1}{2}} X_{4}^{-\frac{1}{2}} (X_4 + X_5)^{-\frac{1}{2}} $\\
Now let us multiply $(F_{1,2,3}^{x})^{(-\frac{1}{2})} = X_{1}^{\frac{1}{2}} X_{2}^{-\frac{1}{2}} (X_2 + X_3)^{-\frac{1}{2}}$ with $(F_{3,4,5}^{x})^{(\frac{1}{2})}=(1 - q^{-\frac{1}{2}})(U_{+} - X_3 - X_4 - X_5) X_{3}^{\frac{1}{2}} X_{4}^{-\frac{1}{2}} (X_4 + X_5)^{-\frac{1}{2}}$ for to have \\
$(F_{1,2,3}^{x})^{(-\frac{1}{2})} (F_{3,4,5}^{x})^{(\frac{1}{2})}$\\
$\hspace*{35pt} = X_{1}^{\frac{1}{2}} X_{2}^{-\frac{1}{2}} (X_2 + X_3)^{-\frac{1}{2}} (1 - q^{-\frac{1}{2}})(U_{+} - X_3 - X_4 - X_5) X_{3}^{\frac{1}{2}} X_{4}^{-\frac{1}{2}} (X_4 + X_5)^{-\frac{1}{2}}$\\
$\hspace*{35pt} = X_{1}^{\frac{1}{2}} X_{2}^{-\frac{1}{2}} (X_2 + X_3)^{-\frac{1}{2}} ((U_{+} - X_3 - X_4 - X_5) X_{3}^{\frac{1}{2}} X_{4}^{-\frac{1}{2}} (X_4 + X_5)^{-\frac{1}{2}}) - q^{-\frac{1}{2}} (U_{+} - X_3 - X_4 - X_5) X_{3}^{\frac{1}{2}} X_{4}^{-\frac{1}{2}} (X_4 + X_5)^{-\frac{1}{2}})$\\
$\hspace*{35pt} =  X_{1}^{\frac{1}{2}} X_{2}^{-\frac{1}{2}} (X_2 + X_3)^{-\frac{1}{2}} (U_{+} - X_3 - X_4 - X_5) X_{3}^{\frac{1}{2}} X_{4}^{-\frac{1}{2}} (X_4 + X_5)^{-\frac{1}{2}} - q^{-\frac{1}{2}} X_{1}^{\frac{1}{2}} X_{2}^{-\frac{1}{2}} (X_2 + X_3)^{-\frac{1}{2}} (U_{+} - X_3 - X_4 - X_5) X_{3}^{\frac{1}{2}} X_{4}^{-\frac{1}{2}} (X_4 + X_5)^{-\frac{1}{2}}$\\
$\hspace*{35pt} =  X_{1}^{\frac{1}{2}} X_{2}^{-\frac{1}{2}} (X_2 + X_3)^{-\frac{1}{2}} U_{+}  X_{3}^{\frac{1}{2}} X_{4}^{-\frac{1}{2}} (X_4 + X_5)^{-\frac{1}{2}}$\\
$\hspace*{35pt} - X_{1}^{\frac{1}{2}} X_{2}^{-\frac{1}{2}} (X_2 + X_3)^{-\frac{1}{2}} X_3  X_{3}^{\frac{1}{2}} X_{4}^{-\frac{1}{2}} (X_4 + X_5)^{-\frac{1}{2}} $\\
$\hspace*{35pt} - X_{1}^{\frac{1}{2}} X_{2}^{-\frac{1}{2}} (X_2 + X_3)^{-\frac{1}{2}} X_4  X_{3}^{\frac{1}{2}} X_{4}^{-\frac{1}{2}} (X_4 + X_5)^{-\frac{1}{2}} $\\
$\hspace*{35pt} - X_{1}^{\frac{1}{2}} X_{2}^{-\frac{1}{2}} (X_2 + X_3)^{-\frac{1}{2}} X_5  X_{3}^{\frac{1}{2}} X_{4}^{-\frac{1}{2}} (X_4 + X_5)^{-\frac{1}{2}}$\\

$\hspace*{35pt} - q^{-\frac{1}{2}} X_{1}^{\frac{1}{2}} X_{2}^{-\frac{1}{2}} (X_2 + X_3)^{-\frac{1}{2}} U_{+}  X_{3}^{\frac{1}{2}} X_{4}^{-\frac{1}{2}} (X_4 + X_5)^{-\frac{1}{2}} $\\
$\hspace*{35pt} - q^{-\frac{1}{2}} X_{1}^{\frac{1}{2}} X_{2}^{-\frac{1}{2}} (X_2 + X_3)^{-\frac{1}{2}} X_{3}  X_{3}^{\frac{1}{2}} X_{4}^{-\frac{1}{2}} (X_4 + X_5)^{-\frac{1}{2}} $\\
$\hspace*{35pt} - q^{-\frac{1}{2}} X_{1}^{\frac{1}{2}} X_{2}^{-\frac{1}{2}} (X_2 + X_3)^{-\frac{1}{2}} X_{4}  X_{3}^{\frac{1}{2}} X_{4}^{-\frac{1}{2}} (X_4 + X_5)^{-\frac{1}{2}} $\\
$\hspace*{35pt} - q^{-\frac{1}{2}} X_{1}^{\frac{1}{2}} X_{2}^{-\frac{1}{2}} (X_2 + X_3)^{-\frac{1}{2}} X_{5}  X_{3}^{\frac{1}{2}} X_{4}^{-\frac{1}{2}} (X_4 + X_5)^{-\frac{1}{2}} $\\
$\hspace*{35pt} = (1 - q^{-\frac{1}{2}})X_{1}^{\frac{1}{2}} X_{2}^{-\frac{1}{2}} (X_2 + X_3)^{-\frac{1}{2}} U_{+}  X_{3}^{\frac{1}{2}} X_{4}^{-\frac{1}{2}} (X_4 + X_5)^{-\frac{1}{2}}$\\
$\hspace*{35pt} + ( 1 - q^{-\frac{1}{2}} )X_{1}^{\frac{1}{2}} X_{2}^{-\frac{1}{2}} (X_2 + X_3)^{-\frac{1}{2}} X_{3}  X_{3}^{\frac{1}{2}} X_{4}^{-\frac{1}{2}} (X_4 + X_5)^{-\frac{1}{2}}$\\
$\hspace*{35pt} + (1 - q^{-\frac{1}{2}} )X_{1}^{\frac{1}{2}} X_{2}^{-\frac{1}{2}} (X_2 + X_3)^{-\frac{1}{2}} X_{4}  X_{3}^{\frac{1}{2}} X_{4}^{-\frac{1}{2}} (X_4 + X_5)^{-\frac{1}{2}}$\\
$\hspace*{35pt} + ( 1 - q^{-\frac{1}{2}}  )X_{1}^{\frac{1}{2}} X_{2}^{-\frac{1}{2}} (X_2 + X_3)^{-\frac{1}{2}} X_{5}  X_{3}^{\frac{1}{2}} X_{4}^{-\frac{1}{2}} (X_4 + X_5)^{-\frac{1}{2}}$\\
And again let us calculate the multiplication:
\begin{equation}
- q^{-\frac{1}{2}} (F_{1,2,3}^{x})^{(\frac{1}{2})} (F_{3,4,5}^{x})^{(-\frac{1}{2})}
\end{equation}
$\hspace*{35pt} = - q^{-\frac{1}{2}} (1 - q^{-\frac{1}{2}})(U_{+} - X_3) X_{1}^{\frac{1}{2}} X_{2}^{-\frac{1}{2}} (X_2 + X_3)^{-\frac{1}{2}} X_{3}^{\frac{1}{2}} X_{4}^{-\frac{1}{2}} (X_4 + X_5)^{-\frac{1}{2}}$\\
$\hspace*{35pt} = - q^{-\frac{1}{2}} (1 - q^{-\frac{1}{2}})U_{+} X_{1}^{\frac{1}{2}} X_{2}^{-\frac{1}{2}} (X_2 + X_3)^{-\frac{1}{2}} X_{3}^{\frac{1}{2}} X_{4}^{-\frac{1}{2}} (X_4 + X_5)^{-\frac{1}{2}}$\\
$\hspace*{35pt} + q^{-\frac{1}{2}} (1 - q^{-\frac{1}{2}}) X_3 X_{1}^{\frac{1}{2}} X_{2}^{-\frac{1}{2}} (X_2 + X_3)^{-\frac{1}{2}} X_{3}^{\frac{1}{2}} X_{4}^{-\frac{1}{2}} (X_4 + X_5)^{-\frac{1}{2}} $\\
$\hspace*{35pt}  = -(1 - q^{-\frac{1}{2}}) X_{1}^{\frac{1}{2}} X_{2}^{-\frac{1}{2}} (X_2 + X_3)^{-\frac{1}{2}} U_{+} X_{3}^{\frac{1}{2}} X_{4}^{-\frac{1}{2}} (X_4 + X_5)^{-\frac{1}{2}}$\\
$\hspace*{35pt} -(1 - q^{-\frac{1}{2}}) X_{1}^{\frac{1}{2}} X_{2}^{-\frac{1}{2}} (X_2 + X_3)^{-\frac{1}{2}} X_{3} X_{3}^{\frac{1}{2}} X_{4}^{-\frac{1}{2}} (X_4 + X_5)^{-\frac{1}{2}} $\\
$\hspace*{35pt} \Rightarrow   (F_{1,2,3}^{x})^{(-\frac{1}{2})} (F_{3,4,5}^{x})^{(\frac{1}{2})} - q^{-\frac{1}{2}} (F_{1,2,3}^{x})^{(\frac{1}{2})} (F_{3,4,5}^{x})^{(-\frac{1}{2})}= (1 - q^{-\frac{1}{2}})X_{1}^{\frac{1}{2}} X_{2}^{-\frac{1}{2}} (X_2 + X_3)^{-\frac{1}{2}} X_{3}^{\frac{1}{2}} \\X_{4}^{-\frac{1}{2}} (X_4 + X_5)^{-\frac{1}{2}} $\\
Lets call it $\rho_{1,5}$. But the coefficients here are not so important for us, so let us skip it and write it as follows:\\
\begin{equation}
\rho_{1,5} = X_{1}^{\frac{1}{2}} X_{2}^{-\frac{1}{2}} (X_2 + X_3)^{-\frac{1}{2}} X_{3}^{\frac{1}{2}} X_{4}^{-\frac{1}{2}} (X_4 + X_5)^{-\frac{1}{2}}
\end{equation}
that is a five points generator of lattice Virasoro algebra, but again we are not interested on it and still looking for a simplest one of type $ABCD$.\\
So let us to define another such kind of solutions as we proposed and experienced already;
\begin{equation}
(F_{2,3,4}^{x})^{(-\frac{1}{2})} := X_{2}^{\frac{1}{2}} X_{3}^{-\frac{1}{2}} (X_3 + X_4)^{-\frac{1}{2}}
\end{equation}
\begin{equation}
(F_{2,3,4}^{x})^{(\frac{1}{2})} := (1 - q^{-\frac{1}{2}}) (U_{+} - X_3 - X_4) X_{2}^{\frac{1}{2}} X_{3}^{-\frac{1}{2}} (X_3 + X_4)^{-\frac{1}{2}}
\end{equation}
and then define;
\begin{equation}
\rho_{1,4} = (F_{1,2,3}^{x})^{(-\frac{1}{2})}(F_{2,3,4}^{x})^{(\frac{1}{2})} - q^{-\frac{1}{2}} (F_{1,2,3}^{x})^{(\frac{1}{2})}(F_{2,3,4}^{x})^{(-\frac{1}{2})}
\end{equation}
Let us calculate it;
\begin{equation}
(F_{1,2,3}^{x})^{(-\frac{1}{2})}(F_{2,3,4}^{x})^{(\frac{1}{2})}
\end{equation}
$\hspace*{35pt} =  X_{1}^{\frac{1}{2}} X_{2}^{-\frac{1}{2}} (X_2 + X_3)^{-\frac{1}{2}} ((1 - q^{-\frac{1}{2}})(U_{+} - X_3 - X_4) X_{2}^{\frac{1}{2}} X_{3}^{-\frac{1}{2}} (X_3 + X_4)^{-\frac{1}{2}}$\\
$\hspace*{35pt}  = (1 - q^{-\frac{1}{2}}) X_{1}^{\frac{1}{2}} X_{2}^{-\frac{1}{2}} (X_2 + X_3)^{-\frac{1}{2}}(U_{+} - X_3 - X_4) X_{2}^{\frac{1}{2}} X_{3}^{-\frac{1}{2}} (X_3 + X_4)^{-\frac{1}{2}}$\\
$\hspace*{35pt}  = (1 - q^{-\frac{1}{2}}) X_{1}^{\frac{1}{2}} X_{2}^{-\frac{1}{2}} (X_2 + X_3)^{-\frac{1}{2}}U_{+} X_{2}^{\frac{1}{2}} X_{3}^{-\frac{1}{2}} (X_3 + X_4)^{-\frac{1}{2}} $\\
$\hspace*{35pt}  - (1 - q^{-\frac{1}{2}}) X_{1}^{\frac{1}{2}} X_{2}^{-\frac{1}{2}} (X_2 + X_3)^{-\frac{1}{2}}X_{3} X_{2}^{\frac{1}{2}} X_{3}^{-\frac{1}{2}} (X_3 + X_4)^{-\frac{1}{2}} $\\
$\hspace*{35pt}  - (1 - q^{-\frac{1}{2}}) X_{1}^{\frac{1}{2}} X_{2}^{-\frac{1}{2}} (X_2 + X_3)^{-\frac{1}{2}}X_{4} X_{2}^{\frac{1}{2}} X_{3}^{-\frac{1}{2}} (X_3 + X_4)^{-\frac{1}{2}} $\\
Now let us calculate the other part as well;
\begin{equation}
- q^{-\frac{1}{2}} (F_{1,2,3}^{x})^{(\frac{1}{2})}(F_{2,3,4}^{x})^{(-\frac{1}{2})}
\end{equation}
$\hspace*{35pt}   =  - q^{-\frac{1}{2}} (1  - q^{-\frac{1}{2}})(U_{+} - X_3)  X_{1}^{\frac{1}{2}} X_{2}^{-\frac{1}{2}} (X_2 + X_3)^{-\frac{1}{2}} X_{2}^{\frac{1}{2}} X_{3}^{-\frac{1}{2}} (X_3 + X_4)^{-\frac{1}{2}}$\\
$\hspace*{35pt}   = - q^{-\frac{1}{2}} (1  - q^{-\frac{1}{2}}) U_{+}  X_{1}^{\frac{1}{2}} X_{2}^{-\frac{1}{2}} (X_2 + X_3)^{-\frac{1}{2}} X_{2}^{\frac{1}{2}} X_{3}^{-\frac{1}{2}} (X_3 + X_4)^{-\frac{1}{2}}$\\
$\hspace*{35pt}   +  q^{-\frac{1}{2}} (1  - q^{-\frac{1}{2}})X_3 X_{1}^{\frac{1}{2}} X_{2}^{-\frac{1}{2}} (X_2 + X_3)^{-\frac{1}{2}} X_{2}^{\frac{1}{2}} X_{3}^{-\frac{1}{2}} (X_3 + X_4)^{-\frac{1}{2}}  $\\
$\hspace*{35pt}  = - (1  - q^{-\frac{1}{2}}) X_{1}^{\frac{1}{2}} X_{2}^{-\frac{1}{2}} (X_2 + X_3)^{-\frac{1}{2}} U_{+} X_{2}^{\frac{1}{2}} X_{3}^{-\frac{1}{2}} (X_3 + X_4)^{-\frac{1}{2}} $\\
$\hspace*{35pt}  + (1  - q^{-\frac{1}{2}})  X_{1}^{\frac{1}{2}} X_{2}^{-\frac{1}{2}} (X_2 + X_3)^{-\frac{1}{2}} X_{3} X_{2}^{\frac{1}{2}} X_{3}^{-\frac{1}{2}} (X_3 + X_4)^{-\frac{1}{2}}$.\\
So we have\\
\begin{equation}
\rho_{1,4} = - (1  - q^{-\frac{1}{2}})  X_{1}^{\frac{1}{2}} X_{2}^{-\frac{1}{2}} (X_2 + X_3)^{-\frac{1}{2}} X_{4} X_{2}^{\frac{1}{2}} X_{3}^{-\frac{1}{2}} (X_3 + X_4)^{-\frac{1}{2}}
\end{equation}
that has degree zero, so it can be one of the generators for lattice Virasoro algebra, but still again is not interested for us, so we should look for a simplest one. And again the coefficient is not important for us, so lets skip it. So we have\\
\begin{equation}
\rho_{1,4} = X_{1}^{\frac{1}{2}} X_{2}^{-\frac{1}{2}} (X_2 + X_3)^{-\frac{1}{2}} X_{4} X_{2}^{\frac{1}{2}} X_{3}^{-\frac{1}{2}} (X_3 + X_4)^{-\frac{1}{2}}
\end{equation}
and
\begin{equation}
\rho_{1,5}^{-1} = (X_4 + X_5)^{\frac{1}{2}}X_{4}^{\frac{1}{2}} X_{3}^{-\frac{1}{2}}(X_4 + X_5)^{-1}(X_2 + X_3)^{\frac{1}{2}}        X_{2}^{\frac{1}{2}}  X_{1}^{-\frac{1}{2}}
\end{equation}
Let us calculate the multiplication $\rho_{1,5}^{-1} \rho_{1,4}$ for to see what will happen?\\
\begin{cl}
	The claim is that $\rho_{1,5}^{-1} \rho_{1,4}$ should give us the answer?
\end{cl}
\begin{proof}
	$\rho_{1,5}^{-1} \rho_{1,4}= (X_4 + X_5)^{\frac{1}{2}}X_{4}^{\frac{1}{2}} X_{3}^{-\frac{1}{2}} (X_4 + X_5)^{-1} X_{4} X_{2}^{\frac{1}{2}} X_{3}^{-\frac{1}{2}} (X_3 + X_4)^{-\frac{1}{2}} $\\
	$\hspace*{35pt}  = (X_4 + X_5)^{\frac{1}{2}}X_{4}^{\frac{1}{2}}  q^{\frac{1}{2}} (X_4 + X_5)^{-1} X_{3}^{-\frac{1}{2}} X_{4} X_{2}^{\frac{1}{2}} X_{3}^{-\frac{1}{2}} (X_3 + X_4)^{-\frac{1}{2}}$\\
	$\hspace*{35pt}  = (X_4 + X_5)^{\frac{1}{2}} q^{-\frac{1}{2}} q^{\frac{1}{2}} (X_4 + X_5)^{-1} X_{4}^{\frac{1}{2}} X_{3}^{-\frac{1}{2}} X_{4}  X_{2}^{\frac{1}{2}} X_{3}^{-\frac{1}{2}} (X_3 + X_4)^{-\frac{1}{2}} $\\
	$\hspace*{35pt}  = (X_4 + X_5)^{-\frac{1}{2}} X_{4}^{\frac{1}{2}} X_{3}^{-\frac{1}{2}} X_{4} X_{2}^{\frac{1}{2}} X_{3}^{-\frac{1}{2}} (X_3 + X_4)^{-\frac{1}{2}}$\\
	$\hspace*{35pt}  = (X_4 + X_5)^{-\frac{1}{2}} X_{4}^{\frac{1}{2}} q^{-\frac{1}{2}} X_{4} X_{3}^{-\frac{1}{2}}  X_{2}^{\frac{1}{2}} X_{3}^{-\frac{1}{2}} (X_3 + X_4)^{-\frac{1}{2}}$\\
	$\hspace*{35pt}  = (X_4 + X_5)^{-\frac{1}{2}} X_{4}^{\frac{1}{2}} q^{-\frac{1}{2}} X_{4} X_{2}^{\frac{1}{2}} q^{\frac{1}{4}} X_{3}^{-1}  (X_3 + X_4)^{-\frac{1}{2}}$\\
	$\hspace*{35pt}  = (X_4 + X_5)^{-\frac{1}{2}} X_{4}^{\frac{1}{2}} X_{2}^{\frac{1}{2}} q^{-\frac{3}{4}} X_{4}  X_{3}^{-1}  (X_3 + X_4)^{-\frac{1}{2}}$\\
	$\hspace*{35pt}  =  q^{-\frac{3}{4}}(X_4 + X_5)^{-\frac{1}{2}} X_{4}^{\frac{1}{2}}  X_{2}^{\frac{1}{2}}   (X_3 + X_4)^{-\frac{1}{2}}$\\
	And as always the coefficients are not important for us, so let us the set the final solution as follows:
	\begin{equation}
	\rho_{1,5}^{-1} \rho_{1,4} = (X_4 + X_5)^{-\frac{1}{2}} X_{4}^{\frac{1}{2}}  X_{2}^{\frac{1}{2}}   (X_3 + X_4)^{-\frac{1}{2}}
	\end{equation}
	as we were looking for.
\end{proof}

Now suppose the following representation of $sl_2$.\\
Define:\\
$$F = \partial_{(U_{+} - X_3 - X_4)}$$
\begin{equation}
H = (U_{+} - X_3 - X_4) \partial_{(U_{+} - X_3 - X_4)} + X_1 \partial_{X_{1}} + X_2 \partial_{X_{2}} + X_3 \partial_{X_{3}}
\end{equation}
$$E = (U_{+} - X_3 - X_4)^{2} \partial_{(U_{+} - X_3 - X_4)} + (X_{1}^{2} + X_1 X_2 + X_1 X_3 + X_1 (U_{+} - X_3 - X_4) )\partial_{X_{1}} $$
$$\hspace*{-30pt} + (X_{2}^{2} + X_2 X_3 + X_2 (U_{+} - X_3 - X_4))\partial_{X_{2}} + (X_{3}^{2} + X_3 (U_{+} - X_3 - X_4) )\partial_{X_{3}}$$
As before set\\
$$(F_{k,k+3}^{x})^{(i)}=(F_{k,k+1,k+2,k+3}^{x})^{(i)}:=[\Sigma^X , [\Sigma^X , [ \cdots , [ \Sigma^X , (F_{k,k+3}^{x})^{(i)}]]\cdots ]]$$
$$(F_{k+1,k+4}^{x})^{(j)}=(F_{k+1,k+2,k+3,k+4}^{x})^{(j)}=[\Sigma^X , [\Sigma^X , [ \cdots , [ \Sigma^X , (F_{k+1,k+4}^{x})^{(j)}]]\cdots ]$$
Set $k=1$ and $i=-\frac{1}{2}$ and set $(F_{1,4}^{x})^{(-\frac{1}{2})} = X_{1}^{\frac{1}{2}}X_{2}^{-\frac{1}{2}}(X_{2} + X_{3}+X_{4})^{-\frac{1}{2}}$.\\
and as what we had already, set
\begin{equation}
(F_{1,4}^{x})^{(\frac{1}{2})}= (1 - q^{(-\frac{1}{2})})(U_{+} - X_3 - X_4)X_{1}^{\frac{1}{2}}X_{2}^{-\frac{1}{2}}(X_{2} + X_{3}+X_{4})^{-\frac{1}{2}}
\end{equation}
where as usual $U_{+} = \Sigma_{i=3}^{+\infty} X_i$.\\
Now as before, set $X_1 \rightarrow X_3$ and $X_2 \rightarrow X_4$ and $X_3 \rightarrow X_5$ and $X_4 \rightarrow X_6$ in $(F_{1,4}^{x})^{(-\frac{1}{2})}$, then we will have,\\
\begin{equation}
(F_{3,6}^{x})^{(-\frac{1}{2})} = (F_{3,4,5,6}^{x})^{(-\frac{1}{2})}= X_{3}^{\frac{1}{2}}X_{4}^{-\frac{1}{2}}(X_{4} + X_{5}+X_{6})^{-\frac{1}{2}}
\end{equation}
\begin{equation}
(F_{3,6}^{x})^{(\frac{1}{2})} = (1 - q^{(-\frac{1}{2})})(U_{+} - X_3 - X_4-X_5 - X_6)X_{3}^{\frac{1}{2}}X_{4}^{-\frac{1}{2}}(X_{4} + X_{5}+X_{6})^{-\frac{1}{2}}
\end{equation}
and then we will proceed as before again
\begin{equation}
(F_{1,4}^{x})^{(-\frac{1}{2})} (F_{3,6}^{x})^{(\frac{1}{2})}
\end{equation}
$ = X_{1}^{\frac{1}{2}}X_{2}^{-\frac{1}{2}}(X_{2} + X_{3}+X_{4})^{-\frac{1}{2}} (1 - q^{(-\frac{1}{2})})(U_{+} - X_3 - X_4-X_5 - X_6)X_{3}^{\frac{1}{2}}X_{4}^{-\frac{1}{2}}(X_{4} + X_{5}+X_{6})^{-\frac{1}{2}}  $\\
$ = (1 - q^{(-\frac{1}{2})}) X_{1}^{\frac{1}{2}}X_{2}^{-\frac{1}{2}}(X_{2} + X_{3}+X_{4})^{-\frac{1}{2}} (U_{+} - X_3- X_4)X_{3}^{\frac{1}{2}}X_{4}^{-\frac{1}{2}}(X_{4} + X_{5}+X_{6})^{-\frac{1}{2}} $\\
$- (1 - q^{(-\frac{1}{2})}) X_{1}^{\frac{1}{2}}X_{2}^{-\frac{1}{2}}(X_{2} + X_{3}+X_{4})^{-\frac{1}{2}} X_5 X_{3}^{\frac{1}{2}}X_{4}^{-\frac{1}{2}}(X_{4} + X_{5}+X_{6})^{-\frac{1}{2}}$\\
$ - (1 - q^{(-\frac{1}{2})}) X_{1}^{\frac{1}{2}}X_{2}^{-\frac{1}{2}}(X_{2} + X_{3}+X_{4})^{-\frac{1}{2}} X_6 X_{3}^{\frac{1}{2}}X_{4}^{-\frac{1}{2}}(X_{4} + X_{5}+X_{6})^{-\frac{1}{2}}$\\
and on the other hand, we have
\begin{equation}
- q^{(-\frac{1}{2})} (F_{1,4}^{x})^{(\frac{1}{2})} (F_{3,6}^{x})^{(-\frac{1}{2})}
\end{equation}
$ = - q^{(-\frac{1}{2})}(1 - q^{(-\frac{1}{2})})(U_{+} - X_3 - X_4) X_{1}^{\frac{1}{2}}X_{2}^{-\frac{1}{2}}(X_{2} + X_{3}+X_{4})^{-\frac{1}{2}} X_{3}^{\frac{1}{2}}X_{4}^{-\frac{1}{2}}(X_{4} + X_{5}+X_{6})^{-\frac{1}{2}}$ \\
$= - (1 - q^{(-\frac{1}{2})}) X_{1}^{\frac{1}{2}}X_{2}^{-\frac{1}{2}}(X_{2} + X_{3}+X_{4})^{-\frac{1}{2}} (U_{+} - X_3 - X_4) X_{3}^{\frac{1}{2}}X_{4}^{-\frac{1}{2}}(X_{4} + X_{5}+X_{6})^{-\frac{1}{2}}$\\
and so
$(F_{1,4}^{x})^{(-\frac{1}{2})} (F_{3,6}^{x})^{(\frac{1}{2})}  - q^{(-\frac{1}{2})} (F_{1,4}^{x})^{(\frac{1}{2})} (F_{3,6}^{x})^{(-\frac{1}{2})} $\\
$= - (1 - q^{(-\frac{1}{2})})  X_{1}^{\frac{1}{2}}X_{2}^{-\frac{1}{2}}(X_{2} + X_{3}+X_{4})^{-\frac{1}{2}} (X_{5}+X_{6}) X_{3}^{\frac{1}{2}}X_{4}^{-\frac{1}{2}}(X_{4} + X_{5}+X_{6})^{-\frac{1}{2}}$\\
as always let us call it
\begin{equation}
\rho_{1,6}=  X_{1}^{\frac{1}{2}}X_{2}^{-\frac{1}{2}}(X_{2} + X_{3}+X_{4})^{-\frac{1}{2}} (X_{4} +X_{5}+X_{6}) X_{3}^{\frac{1}{2}}X_{4}^{-\frac{1}{2}}(X_{4} + X_{5}+X_{6})^{-\frac{1}{2}}
\end{equation}
\begin{equation}
\rho_{1,6}^{-1}=(X_{4} + X_{5}+X_{6})^{\frac{1}{2}}X_{4}^{\frac{1}{2}}X_{3}^{-\frac{1}{2}}(X_{4} +X_{5}+X_{6})^{-1}                            (X_{2} + X_{3}+X_{4})^{\frac{1}{2}}X_{2}^{\frac{1}{2}} X_{1}^{-\frac{1}{2}}
\end{equation}
Now again let us define another such kind of solutions:
\begin{equation}
(F_{2,5}^{x})^{(-\frac{1}{2})} = (F_{2,3,4,5}^{x})^{(-\frac{1}{2})}:= X_{2}^{\frac{1}{2}}X_{3}^{-\frac{1}{2}}(X_{3} + X_{4}+X_{5})^{-\frac{1}{2}}
\end{equation}
\begin{equation}
(F_{2,5}^{x})^{(\frac{1}{2})} = (F_{2,3,4,5}^{x})^{(\frac{1}{2})}:=(1 - q^{(-\frac{1}{2})})(U_{+} - X_3 - X_4- X_5) X_{2}^{\frac{1}{2}}X_{3}^{-\frac{1}{2}}(X_{3} + X_{4}+X_{5})^{-\frac{1}{2}}
\end{equation}
and we are looking for the value of the following objects;
\begin{equation}
(F_{1,4}^{x})^{(-\frac{1}{2})} (F_{2,5}^{x})^{(\frac{1}{2})}
\end{equation}
$=X_{1}^{\frac{1}{2}}X_{2}^{-\frac{1}{2}}(X_{2} + X_{3}+X_{4})^{-\frac{1}{2}} ((1 - q^{(-\frac{1}{2})})(U_{+} - X_3 - X_4- X_5) X_{2}^{\frac{1}{2}}X_{3}^{-\frac{1}{2}}(X_{3} + X_{4}+X_{5})^{-\frac{1}{2}} $\\
$= ((1 - q^{(-\frac{1}{2})}) X_{1}^{\frac{1}{2}}X_{2}^{-\frac{1}{2}}(X_{2} + X_{3}+X_{4})^{-\frac{1}{2}} (U_{+} - X_3 - X_4) X_{2}^{\frac{1}{2}}X_{3}^{-\frac{1}{2}}(X_{3} + X_{4}+X_{5})^{-\frac{1}{2}}$\\
$- ((1 - q^{(-\frac{1}{2})}) X_{1}^{\frac{1}{2}}X_{2}^{-\frac{1}{2}}(X_{2} + X_{3}+X_{4})^{-\frac{1}{2}} X_5 X_{2}^{\frac{1}{2}}X_{3}^{-\frac{1}{2}}(X_{3} + X_{4}+X_{5})^{-\frac{1}{2}} $\\
and on the other side we have:
\begin{equation}
- q^{(-\frac{1}{2})} (F_{1,4}^{x})^{(\frac{1}{2})} (F_{2,5}^{x})^{(-\frac{1}{2})}
\end{equation}
$= - q^{(-\frac{1}{2})} (1 - q^{(-\frac{1}{2})}) (U_{+} - X_3 - X_4)X_{1}^{\frac{1}{2}}X_{2}^{-\frac{1}{2}}(X_{2} + X_{3}+X_{4})^{-\frac{1}{2}} X_{2}^{\frac{1}{2}}X_{3}^{-\frac{1}{2}}(X_{3} + X_{4}+X_{5})^{-\frac{1}{2}} $\\
$= - (1 - q^{(-\frac{1}{2})}) X_{2}^{-\frac{1}{2}}(X_{2} + X_{3}+X_{4})^{-\frac{1}{2}} (U_{+} - X_3 - X_4) X_{2}^{\frac{1}{2}}X_{3}^{-\frac{1}{2}}(X_{3} + X_{4}+X_{5})^{-\frac{1}{2}}  $
so we have:
$\rho_{1,5} = (F_{1,4}^{x})^{(-\frac{1}{2})} (F_{2,5}^{x})^{(\frac{1}{2})} - q^{(-\frac{1}{2})} (F_{1,4}^{x})^{(\frac{1}{2})} (F_{2,5}^{x})^{(-\frac{1}{2})}$
\begin{equation}
\rho_{1,5} = X_{1}^{\frac{1}{2}}X_{2}^{-\frac{1}{2}}(X_{2} + X_{3}+X_{4})^{-\frac{1}{2}} X_5 X_{2}^{\frac{1}{2}}X_{3}^{-\frac{1}{2}}(X_{3} + X_{4}+X_{5})^{-\frac{1}{2}}
\end{equation}
\begin{cl}
	$\rho_{1,6}^{-1} \rho_{1,5}$ has degree zero.
\end{cl}
\begin{proof}
	$\rho_{1,6}^{-1} \rho_{1,5} = (X_{4} + X_{5}+X_{6})^{\frac{1}{2}}X_{4}^{\frac{1}{2}}X_{3}^{-\frac{1}{2}}(X_{4} +X_{5}+X_{6})^{-1} (X_{2} + X_{3}+X_{4})^{\frac{1}{2}}X_{2}^{\frac{1}{2}} X_{1}^{-\frac{1}{2}} X_{1}^{\frac{1}{2}}\\X_{2}^{-\frac{1}{2}}(X_{2} + X_{3}+X_{4})^{-\frac{1}{2}} X_5 X_{2}^{\frac{1}{2}}X_{3}^{-\frac{1}{2}}(X_{3} + X_{4}+X_{5})^{-\frac{1}{2}}$\\
	$= (X_{4} + X_{5}+X_{6})^{\frac{1}{2}}X_{4}^{\frac{1}{2}}X_{3}^{-\frac{1}{2}}(X_{4} +X_{5}+X_{6})^{-1}X_5 X_{2}^{\frac{1}{2}}X_{3}^{-\frac{1}{2}}(X_{3} + X_{4}+X_{5})^{-\frac{1}{2}}$\\
	$= q^{(-\frac{3}{4})}(X_{4} + X_{5}+X_{6})^{-\frac{1}{2}}X_{4}^{\frac{1}{2}} X_{2}^{\frac{1}{2}}(X_{3} + X_{4}+X_{5})^{-\frac{1}{2}}$
\end{proof}
So we have: \\
\begin{equation}
\rho_{1,6}^{-1} \rho_{1,5} = (X_{4} + X_{5}+X_{6})^{-\frac{1}{2}}X_{4}^{\frac{1}{2}} X_{2}^{\frac{1}{2}}(X_{3} + X_{4}+X_{5})^{-\frac{1}{2}}
\end{equation}

\section{Conclusion}

\vspace*{10pt}
Four point invariant, that's coming from the 3-dimensional representation of $sl_2$;\\
$$[\Sigma_{-\infty}^{+\infty} X_i , (X_4 + X_5)^{-1} X_4 X_2 (X_3 + X_4)^{-1}]_q = 0$$
Five point invariant, that's coming from the 4-dimensional representation of $sl_2$;\\
$$[\Sigma_{-\infty}^{+\infty} X_i , (X_4 + X_5+ X_6)^{-1} X_4 X_2 (X_3 + X_4+ X_5)^{-1}]_q = 0$$
\begin{cl}
	We have the following n-point invariant that's coming from the n-dimensional representation.
	$$(X_4 + \cdots + X_{n})^{-1} X_4 X_2 (X_3 + \cdots + X_{n-1})^{-1}$$
\end{cl}
And then by using the shift operators (\ref{Equ5}), we will have the space of all nontrivial generators of lattice Virasoro algebra.\\
We call these kind of generators that are the only nontrivial ones:
$$Generators ~ of ~ type ~ "ABCD" ~ $$

These (new lattice) algebras are so important and may in principle lead to a new integrable chain equations which people can hardly provide.

Now let us check the satisfactory of our generators in quantum Serre relations.\\
We need to show the correctness of
$(X_2 + X_3 + X_4 + X_5)((X_4 + X_5)^{-\frac{1}{2}}X_{4}^{\frac{3}{2}} X_{2}^{\frac{1}{2}}  X_{3}^{-1} (X_3 + X_4)^{\frac{1}{2}}) \stackrel{?}{=} ((X_4 + X_5)^{-\frac{1}{2}}X_{4}^{\frac{3}{2}} X_{2}^{\frac{1}{2}}  X_{3}^{-1} (X_3 + X_4)^{-\frac{1}{2}})(X_2 + X_3 + X_4 + X_5) $\\
So let us proceed it as before on each component:\\
$X_2 (X_4 + X_5)^{-\frac{1}{2}}X_{4}^{\frac{3}{2}} X_{2}^{\frac{1}{2}}  X_{3}^{-1} (X_3 + X_4)^{\frac{1}{2}}\stackrel{?}{=} (X_4 + X_5)^{-\frac{1}{2}}X_{4}^{\frac{3}{2}} X_{2}^{\frac{1}{2}}  X_{3}^{-1} (X_3 + X_4)^{-\frac{1}{2}} X_2$\\
$(X_4 + X_5)^{\frac{1}{2}} X_2 (X_4 + X_5)^{-\frac{1}{2}} X_{4}^{\frac{3}{2}} X_{2}^{\frac{1}{2}}  X_{3}^{-1} \stackrel{?}{=}  X_{4}^{\frac{3}{2}} X_{2}^{\frac{1}{2}}  X_{3}^{-1} (X_3 + X_4)^{-\frac{1}{2}} X_2(X_3 + X_4)^{\frac{1}{2}}  $\\
$X_2 X_{4}^{\frac{3}{2}} X_{2}^{\frac{1}{2}} X_{3}^{-1} \stackrel{?}{=} q^{\frac{1}{2}} X_{4}^{\frac{3}{2}} X_{2}^{\frac{1}{2}}  X_{3}^{-1} X_2 $\\
$X_2 X_{4}^{\frac{3}{2}} X_{2}^{\frac{1}{2}} X_{3}^{-1} \stackrel{?}{=} q^{\frac{1}{2}} q X_{4}^{\frac{3}{2}} X_2 X_{2}^{\frac{1}{2}} X_{3}^{-1}$\\
$X_2 X_{4}^{\frac{3}{2}} X_{2}^{\frac{1}{2}} X_{3}^{-1} \stackrel{?}{=} q^{\frac{1}{2}} q X_{4}^{\frac{3}{2}} X_2 X_{2}^{\frac{1}{2}} X_{3}^{-1}$\\
$X_2 X_{4}^{\frac{3}{2}} X_{2}^{\frac{1}{2}} X_{3}^{-1} = q^{\frac{1}{2}} q  q^{-\frac{3}{2}}X_{4}^{\frac{3}{2}} X_2 X_{2}^{\frac{1}{2}} X_{3}^{-1}$\\
$X_2 X_{4}^{\frac{3}{2}} X_{2}^{\frac{1}{2}} X_{3}^{-1} = X_2 X_{4}^{\frac{3}{2}}  X_{2}^{\frac{1}{2}} X_{3}^{-1}$\\

Lets do it for $X_3$;\\
$X_3 (X_4 + X_5)^{-\frac{1}{2}}X_{4}^{\frac{3}{2}} X_{2}^{\frac{1}{2}}  X_{3}^{-1} (X_3 + X_4)^{\frac{1}{2}}\stackrel{?}{=} (X_4 + X_5)^{-\frac{1}{2}}X_{4}^{\frac{3}{2}} X_{2}^{\frac{1}{2}}  X_{3}^{-1} (X_3 + X_4)^{-\frac{1}{2}} X_3$\\
$q^{-\frac{1}{2}} X_3 X_{4}^{\frac{3}{2}} X_{2}^{\frac{1}{2}}  X_{3}^{-1}  \stackrel{?}{=} q^{-\frac{1}{2}} X_{4}^{\frac{3}{2}} X_{2}^{\frac{1}{2}}  X_{3}^{-1} X_3 $\\
$q^{-\frac{1}{2}} X_3 X_{4}^{\frac{3}{2}} X_{2}^{\frac{1}{2}}  X_{3}^{-1}  \stackrel{?}{=} q^{-\frac{1}{2}}q^{\frac{1}{2}} X_{4}^{\frac{3}{2}}  X_3 X_{2}^{\frac{1}{2}}  X_{3}^{-1} $\\
$q^{-\frac{1}{2}} X_3 X_{4}^{\frac{3}{2}} X_{2}^{\frac{1}{2}}  X_{3}^{-1}  \stackrel{?}{=} q^{-\frac{1}{2}}  X_3 X_{4}^{\frac{3}{2}}  X_{2}^{\frac{1}{2}}  X_{3}^{-1} $\\
And after repeating a similar trend for $X_4$ and $X_5$ we will get the desired result.

\subsection*{Acknowledgments}  It is a great pleasure to thank Professor  Boris Feigin for to suggesting this interesting problem to me and enlightening discussions during the preparation for this thesis and also I wish to thank Professor Yaroslav Pugai for his wonderful article and his very useful advices and comments.


\begin{thebibliography}{99} 
	\bibitem{1}
	{\em A. Klimyk and K. Schmudgen. Quantum Groups and Their Representations}. Texts and
	Monographs in Physics. Springer, New York et al., 1997.
	
	\bibitem{2}
	{\em A. Alekseev, L. Faddeev and M. Semenov-Tian-Shansky, Hidden quantum groups inside Kac-Moody algebras}, LOMI preprint 1991.
	
	\bibitem{3}
	{\em B. A. Kupershmidt and G. Wilson, Conservation laws and symmetries of generalized sine-Gordon equations}, Comm. Math. Phys., {\bf 81} (1981), 189--202.
	
	\bibitem{4}
	{\em C. Kassel. Quantum Groups}, volume 155 of Graduate texts in mathematics. Springer, New
	York et al., 1995.
	
	\bibitem{5}
	{\em E.K. Sklyanin. On an algebra generated by quadratic relations}. Uspekhi Mat. Nauk, 242:214,
	1985.
	
	\bibitem{6}
	{\em Feigin, B.L.}: talk at RIMS 1992.
	
	\bibitem{7}
	{\em Frenkel, Edward. "Free field realizations in representation theory and conformal field theory."} Proceedings of the International Congress of Mathematicians. Birkhauser Basel, 1995.
	
	\bibitem{8}
	{\em I. Gelfand and L. Dickey, A Lie algebra structure in a formal variations calculus},
	
	\bibitem{9}
	{\em Iohara, Kenji, and Feodor Malikov. ''Rings of skew polynomials and Gel'fand-Kirillov conjecture for quantum groups.''} Communications in mathematical physics {\bf 164.2} (1994): 217--237.
	
	\bibitem{10}
	{\em K. A. Brown and K. R. Goodearl. Lectures on Algebraic Quantum Groups}. Advanced
	Courses in Mathematics, CRM Barcelona. Birkhauser, Basel, Boston and Berlin, 2002.
	
	\bibitem{11}
	{\em Kac V.G.: Infinite-dimensional Lie algebras}. Cambridge University Press
	1990.
	
	\bibitem{12}
	{\em Laugwitz, Robert. "Quantum Groups and Their Representations."}
	
	\bibitem{13}
	{\em Pugay, Ya P. "Lattice W-algebras and quantum groups."} Theoretical and Mathematical Physics {\bf 100.1} (1994): 900--911.
	
	\bibitem{14}
	{\em Stanley, Richard P. "Enumerative Combinatorics Volume 1 second edition."}
	
	
	
\end{thebibliography}
\end{document}